\documentclass[10pt,a4paper]{article}
\usepackage[latin2]{inputenc}
\usepackage[T1]{fontenc}
\usepackage{amsmath,amssymb,multicol,amsthm}
\usepackage{textcomp,array,amsfonts}
\usepackage{graphicx,color,mathrsfs,txfonts}
\usepackage[all]{xy}
\theoremstyle{plain}
\newtheorem{thm}{\bfseries Theorem}[section]
\newtheorem{lem}[thm]{\bfseries Lemma}

\newtheorem{rem}[thm]{\bfseries Remark}
\newtheorem{prop}[thm]{\bfseries Proposition}

\date{}
\title{Tangent to Bloch-Suslin and Grassmannian Complexes over the dual numbers}
\author{Raziuddin Siddiqui \thanks{email: rdsiddiqui@fuuast.edu.pk} \\ \small \textit{Mathematical Sciences Research Centre}\\ 
\small \textit{Federal Urdu University, Karachi.} }
\begin{document}
\maketitle
\begin{abstract}
In this article, we extend Siegel's cross-ratio identity for $2\times2$ determinants over the truncated polynomial ring $F[\varepsilon]_\nu:=F[\varepsilon]/\varepsilon^\nu$. We compute cross-ratios and Goncharov's triple-ratios in $F[\varepsilon]_2$ and $F[\varepsilon]_3$ and use them extensively in our computations for the tangential complexes. We also verify a ''projected five-term'' relation in the group $T\mathcal{B}_2(F)$ which is crucial to prove one of our central statements that describe the morphisms between tangent complex  and Grassmannian complex
\end{abstract}

\textbf{Keywords:} Tangent complex, dual numbers, cross-ratio identity, triple ratio
\section{Introduction:}
At first, 
we present an analogue to Siegel cross-ratio identity \cite{Sieg} for $2\times2$-determinants $\Delta(l^*_i,l^*_j),0\leq i<j\leq3$ for vectors in $(l^*_0,\ldots,l^*_3)\in C_4(\varmathbb{A}^2_{F[\varepsilon]_{\nu}})$ where $C_4(\varmathbb{A}^2_{F[\varepsilon]_{\nu}})$ is the free abelian group generated by the configurations of four vectors in two dimensional affine space over the truncated polynomial ring $F[\varepsilon]_{\nu}$ (see Lemma \ref{lemid} and equations (\ref{genid}) and (\ref{newid})) which is the analogue of (\ref{2did}), and consider their cross-ratios as an element over the truncated polynomial ring $F[\varepsilon]_{\nu}$, i.e., we consider the following cross ratio, $\emph{\textbf{r}}(l^*_0,\ldots,l^*_3)=\left(r_{\varepsilon^0}\varepsilon^0+r_{\varepsilon^1}\varepsilon^1+\cdots+r_{\varepsilon^{\nu-1}}\varepsilon^{\nu-1}\right)(l^*_0,\ldots,l^*_3)$, where $r_{\varepsilon^0}$ is the usual cross-ratio of four points in $\varmathbb{A}^2_F$, while the other elements of $\emph{\textbf{r}}$ are computed in $\S$\ref{cr_Feps}. We introduce a similar construction for the triple-ratio as well (see $\S$\ref{gen_cross}).

Due to this analogue of cross-ratios, we are able to find morphisms between the Grassmannian subcomplex $C_*(\varmathbb{A}^n_{F[\varepsilon]_2},d)$ for $n=2,3$ and the tangent complexes to the Bloch-Suslin and the Goncharov complexes (see $\S$\ref{tbsdb} and $\S$\ref{tbstb}). We also produce result for the projected five-term relation in $T\mathcal{B}_2(F)$ (see Lemma \ref{5pt}) which is analogous to  Goncharov's projected five-term relation in $\mathcal{B}_2(F)$ (see Lemma 2.18 of \cite{Gonc}) and very helpful for the proof of our main result (\ref{claim4b}). 


In $\S$\ref{tbstb}, we provide a possible definition of a group $T\mathcal{B}_3(F)$ which was first defined hypothetically in $\S$9 of \cite{Cath3}. On the basis of our definition, we mimic this construction  with the $F$-vector space $\beta_3^D(F)$ (\cite{SiDs}) and reproduce Cathelineau's 22-term functional equation for $T\mathcal{B}_3(F)$.

Here, we will discuss and try to write geometric configurations for the tangent complex to Bloch-Suslin complex and to Goncharov's complex (see $\S$3 of \cite{Gonc1}). This article will also introduce cross-ratios and identities of determinants for the configurations of vectors in 
$C_m(\varmathbb{A}^{n}_{F[\varepsilon]_\nu})$ for $n=2,3$, $m=3,\ldots,7$ and $\nu\geq1$ ($\nu=1$ is the usual case see \cite{SiDs} and $\S$\ref{con_w_cr} below). 

One of our main results is Theorem \ref{claim4b}. In its proof we shall use combinatorial techniques and will rewrite the triple-ratio as the product of two projected cross-ratios in $F[\varepsilon]_2$.
\section{Configurations of points in  $C_{m}(\varmathbb{A}^n_{F[\varepsilon]_{\nu}})$}\label{con_w_cr}
Let $F$ be a field of characteristic 0. For $\nu\geq1$, we denote the $\nu$th truncated polynomial ring over $F$ by  $F[\varepsilon]_\nu:=F[\varepsilon]/\varepsilon^\nu$. Further define the abelian group $C_m(\varmathbb{A}^n_{F[\varepsilon]_\nu})$ generated by $m$ points in $\varmathbb{A}^n_{F[\varepsilon]_\nu}$ in generic position, where $\varmathbb{A}^n_{F[\varepsilon]_\nu}$ is the $n$-dimensional affine space over the truncated polynomial ring  $F[\varepsilon]_\nu$. We will not consider here degenerate points and we are assuming that no two points coinciding and no three points are lying on a line. Now for the case $n=2$ and $\nu=2$, any $l_i=\left(\begin{array}{c} a_i\\b_i\end{array}\right)\in \varmathbb{A}^2_{F}\setminus\left\lbrace\left(\begin{array}{c} 0\\0\end{array}\right)\right\rbrace$ and $l_{i,\varepsilon}:=\left(\begin{array}{c} a_{i,\varepsilon}\\b_{i,\varepsilon}\end{array}\right)\in \varmathbb{A}^2_{F}$,                                                                                                                                                                                                                                                                                                                                                             we put $l^*_i=\left(\begin{array} {c}a_i+a_{i,\varepsilon}\varepsilon\\ b_i+b_{i,\varepsilon}\varepsilon\end{array}\right)=\left(\begin{array} {c}a_i\\ b_i\end{array}\right)+\left(\begin{array} {c}a_{i,\varepsilon}\\ b_{i,\varepsilon}\end{array}\right)\varepsilon=l_i+l_{i,\varepsilon}\varepsilon$ and define a differential 
\[d:C_{m+1}(\varmathbb{A}^2_{F[\varepsilon]_2})\rightarrow C_m(\varmathbb{A}^2_{F[\varepsilon]_2})\]
\[d:(l^*_0,\ldots,l^*_m)\mapsto \sum_{i=0}^{m}(-1)^i(l^*_0,\ldots,\hat{l}^*_i,\ldots,l^*_m).\]
Let $\omega\in V^*_2$ be a volume element formed in $V_2:=\varmathbb{A}^2_F$ and $\Delta(l_i,l_j)=\langle \omega,l_i\wedge l_j\rangle$, where $l_i,l_j\in\varmathbb{A}^2_{F}$. Here we define
\[{\Delta(l^*_i,l^*_j)}=\Delta(l^*_i,l^*_j)_{\varepsilon^0}+\Delta(l^*_i,l^*_j)_{\varepsilon^1}\varepsilon\]
where
\[{\Delta(l^*_i,l^*_j)}_{\varepsilon^0}=\Delta(l_i,l_j)\quad\text{ and }\quad\Delta(l^*_i,l^*_j)_{\varepsilon^1}=\Delta(l_i,l_{j,\varepsilon})+\Delta(l_{i,\varepsilon},l_j);\]
more generally for $\nu=n+1$, we have
\[l^*_i=l_i+l_{i,\varepsilon}\varepsilon+l_{i,\varepsilon^2}\varepsilon^2+\cdots+l_{i,\varepsilon^n}\varepsilon^n\quad\text{ and } \quad l_{i,\varepsilon^0}=l_i\]
and we get
\[\Delta(l^*_i,l^*_j)=\Delta(l_i,l_j)+\Delta(l^*_i,l^*_j)_{\varepsilon}\varepsilon+\Delta(l^*_i,l^*_j)_{\varepsilon^2}\varepsilon^2+\cdots+\Delta(l^*_i,l^*_j)_{\varepsilon^n}\varepsilon^n,\]
where
\[\Delta(l^*_i,l^*_j)_{\varepsilon^n}=\Delta(l_i,l_{j,\varepsilon^n})+\Delta(l_{i,\varepsilon},l_{j,\varepsilon^{n-1}})+\cdots+\Delta(l_{i,\varepsilon^n},l_j)\]
Consider the Siegel cross-ratio identity for the $2\times2$ determinants of four vectors in $C_4(\varmathbb{A}^2_F)$ (see \cite{Gonc}, \cite{Sieg} or Remark 2 on p155 of \cite{LeNi})
\begin{equation}\label{2did}
\Delta(l_0,l_1)\Delta(l_2,l_3)=\Delta(l_0,l_2)\Delta(l_1,l_3)-\Delta(l_0,l_3)\Delta(l_1,l_2)
\end{equation}
With the above notation, an analogous to Siegel cross-ratio identity turns out to be true for $\varmathbb{A}^2_{F[\varepsilon]_{n+1}}$, and we can  extract further results which are essential for the proof of our main results. Throughout this section we will assume that $\Delta(l_i,l_j)\neq0$ for $i\neq j$.

\begin{lem}\label{lemid}
For $(l^*_0,l^*_1,l^*_2,l^*_3) \in C_4(\varmathbb{A}^2_{F[\varepsilon]_{n+1}})$, we have 
\begin{equation}\label{genid}
\Delta(l^*_0,l^*_1)\Delta(l^*_2,l^*_3)=\Delta(l^*_0,l^*_2)\Delta(l^*_1,l^*_3)-\Delta(l^*_0,l^*_3)\Delta(l^*_1,l^*_2)
\end{equation}
where \[l^*_i=l_i+l_{i,\varepsilon}\varepsilon+l_{i,\varepsilon^2}\varepsilon^2+\cdots+l_{i,\varepsilon^n}\varepsilon^n\quad\text{ and } \quad l_{i,\varepsilon^0}=l_i\]
\[\Delta(l^*_i,l^*_j)=\Delta(l_i,l_j)+\Delta(l^*_i,l^*_j)_{\varepsilon}\varepsilon+\Delta(l^*_i,l^*_j)_{\varepsilon^2}\varepsilon^2+\cdots+\Delta(l^*_i,l^*_j)_{\varepsilon^n}\varepsilon^n\]
where
\[\Delta(l^*_i,l^*_j)_{\varepsilon^n}=\Delta(l_i,l_{j,\varepsilon^n})+\Delta(l_{i,\varepsilon},l_{j,\varepsilon^{n-1}})+\cdots+\Delta(l_{i,\varepsilon^n},l_j)\]
\end{lem}
\begin{proof}
For $r=0,\ldots,n$,  we can write $l^*=\left(\begin{array}{c}
                                             \sum_{r\geq0}l_r\varepsilon^r\\ \sum_{r\geq0}l'_r\varepsilon^r
                                            \end{array}\right)$ and $m^*=\left(\begin{array}{c}
                                             \sum_{r\geq0}m_r\varepsilon^r\\ \sum_{r\geq0}m'_r\varepsilon^r
                                            \end{array}\right).$
                                            
Now we have
\begin{align*}
\Delta(l^*,m^*)=&\left|\begin{array}{cc}
                      \sum_{r\geq0}l_r\varepsilon^r& \sum_{r\geq0}m_r\varepsilon^r\\\sum_{r\geq0}l'_r\varepsilon^r&\sum_{r\geq0}m'_r\varepsilon^r
                      \end{array}\right|=\sum_{r\geq0}\left(\sum_{k=0}^r l_km'_{r-k}-\sum_{k=0}^rl'_km_{r-k}\right)\varepsilon^r\\
                      &=\sum_{r\geq0}\left(\sum_{k=0}^r\Delta\left(l_k,m_{r-k}\right)\right)\varepsilon^r
\end{align*}
Hence
\begin{align*}
\Delta(l^*_0,l^*_1)\Delta(l^*_2,l^*_3)=&\sum_{r\geq0}\left(\sum_{k=0}^r\Delta(l_{0,k},l_{1,r-k})\right)\varepsilon^r\cdot \sum_{s\geq0}\left(\sum_{j=0}^r\Delta(l_{0,j},l_{1,r-j})\right)\varepsilon^s\\
=&\sum_{t\geq0}\varepsilon^t\left(\sum_{r=0}^t\left(\sum_{k=0}^r\Delta(l_{0,k},l_{1,r-k})\sum_{j=0}^{t-r}\Delta(l_{2,j},l_{3,t-r-j})\right)\right)\\
=&\sum_{t\geq0}\varepsilon^t\left(\sum_{r=0}^t\left(\sum_{k=0}^r\sum_{j=0}^{t-r}\Delta(l_{0,k},l_{1,r-k})\Delta(l_{2,j},l_{3,t-r-j})\right)\right),
\end{align*}
and similarly for $\Delta(l^*_0,l^*_2)\Delta(l^*_1,l^*_3)$ and $\Delta(l^*_0,l^*_3)\Delta(l^*_1,l^*_2)$. Hence we use validity of (\ref{2did}) to deduce the analogue for $\Delta(l^*_i,l^*_j)$'s in place of $\Delta(l_i,l_j)$ passing from the ring $F\llbracket \varepsilon\rrbracket$ of power series to a truncated polynomial ring, say to $F[\varepsilon]_{n+1}$.
\end{proof}

As special cases we find for $n=0$ the identity \eqref{2did} while for $n=1$ we have the following identity which will be used extensively below:
\begin{align}\label{newid}
&\Delta(l_0,l_1)\Delta(l^*_2,l^*_3)_{\varepsilon}+\Delta(l_2,l_3)\Delta(l^*_0,l^*_1)_{\varepsilon}\notag\\
&=\left\lbrace\Delta(l_0,l_2)\Delta(l^*_1,l^*_3)_{\varepsilon}+\Delta(l_1,l_3)\Delta(l^*_0,l^*_2)_{\varepsilon}\right\rbrace-\left\lbrace\Delta(l_0,l_3)\Delta(l^*_1,l^*_2)_{\varepsilon}+\Delta(l_1,l_2)\Delta(l^*_0,l^*_3)_{\varepsilon}\right\rbrace.
\end{align}
if we write 
\[(ab)_{\varepsilon^n}:=a_{\varepsilon^n}b_{\varepsilon^0}+a_{\varepsilon^{n-1}}b_{\varepsilon}+\cdots+a_{\varepsilon^0}b_{\varepsilon^n}\]
then (\ref{newid}) can be more concisely written as
\[\left\lbrace\Delta(l^*_0,l^*_1)\Delta(l^*_2,l^*_3)\right\rbrace_{\varepsilon}=\left\lbrace\Delta(l^*_0,l^*_2)\Delta(l^*_1,l^*_3)\right\rbrace_{\varepsilon}-\left\lbrace\Delta(l^*_0,l^*_3)\Delta(l^*_1,l^*_2)\right\rbrace_{\varepsilon}.\]
\subsection{Cross-ratio in $F[\varepsilon]_{\nu}$:}\label{cr_Feps}
In this section we will try to find the cross-ratio of four points in $F[\varepsilon]_{\nu}$ for $\nu=n+1$. We will use the same technique here as we did for the identity (\ref{genid}) but the procedure here involves lengthy calculations. First we define the cross-ratio of four points $(l^*_0,\ldots,l^*_3) \in C_4(\varmathbb{A}^2_{F[\varepsilon]_{n+1}})$ as
\[\textbf{r}(l^*_0,\ldots,l^*_3)=\frac{\Delta(l^*_0,l^*_3)\Delta(l^*_1,l^*_2)}{\Delta(l^*_0,l^*_2)\Delta(l^*_1,l^*_3)}\]
If we expand $\textbf{r}(l^*_0,\ldots,l^*_3)$ as a truncated polynomial over $F[\varepsilon]_{n+1}$, then
\begin{align}\label{trunc_poly}
\textbf{r}(l^*_0,\ldots,l^*_3)=\left(r_{\varepsilon^0}+r_{\varepsilon}\varepsilon+r_{\varepsilon^2}\varepsilon^2+\cdots+r_{\varepsilon^n}\varepsilon^n\right)(l^*_0,\ldots,l^*_3)
\end{align}
If we truncate this for $n=0$, then 
\begin{align}\label{reps0}
\textbf{r}(l^*_0,\ldots,l^*_3)=r_{\varepsilon^0}(l^*_0,\ldots,l^*_3)=r(l_0,\ldots,l_3)=\frac{\Delta(l_0,l_3)\Delta(l_1,l_2)}{\Delta(l_0,l_2)\Delta(l_1,l_3)}
\end{align}
If we truncate \eqref{trunc_poly} for $n=1$ then the coefficient of $\varepsilon^0$ will remain the same as for $n=0$ and we compute the coefficient of $\varepsilon$ in the following way:

Consider $(l^*_0,\ldots,l^*_3) \in C_4(\varmathbb{A}^2_{F[\varepsilon]_2})$ in generic position, we get 
\[\textbf{r}(l^*_0,\ldots,l^*_3)=\frac{\Delta(l^*_0,l^*_3)\Delta(l^*_1,l^*_2)}{\Delta(l^*_0,l^*_2)\Delta(l^*_1,l^*_3)}=\frac{\{\Delta(l_0,l_3)+\Delta(l^*_0,l^*_3)_{\varepsilon}\varepsilon\}\{\Delta(l_1,l_2)+\Delta(l^*_1,l^*_2)_\varepsilon\varepsilon\}}{\{\Delta(l_0,l_2)+\Delta(l^*_0,l^*_2)_{\varepsilon}\varepsilon\}\{\Delta(l_1,l_3)+\Delta(l^*_1,l^*_3)_{\varepsilon}\varepsilon\}}\]
If $a\neq0\in F$ then the inverse of $(a+b\varepsilon)\in F[\varepsilon]_2$ is $\frac{1}{a}-\frac{b}{a^2}\varepsilon \in F[\varepsilon]_2$ (this is the same as the inversion relation in $T\mathcal{B}_2(F)$ discussed later in $\S$\ref{tbsc2}).

Simplify the above by multiplying the inverses of denominators and separate the coefficients of $\varepsilon^0$ and $\varepsilon$. The coefficient of $\varepsilon$ is the following 
\begin{equation}\label{reps1}
r_{\varepsilon}(l^*_0,\ldots,l^*_3)=\frac{\{\Delta(l^*_0,l^*_3)\Delta(l^*_1,l^*_2)\}_{\varepsilon}}{\Delta(l_0,l_2)\Delta(l_1,l_3)}-r(l_0,\ldots,l_3)\frac{\{\Delta(l^*_0,l^*_2)\Delta(l^*_1,l^*_3)\}_{\varepsilon}}{\Delta(l_0,l_2)\Delta(l_1,l_3)}
\end{equation}
Now for $n=2$, i.e.
$(l^*_0,\ldots,l^*_3) \in C_4(\varmathbb{A}^2_{F[\varepsilon]_3})$, we will use $(l_i,l_j)$ instead of $\Delta(l_i,l_j)$ to get
\[\textbf{r}(l^*_0,\ldots,l^*_3)=\frac{\{(l_0,l_3)+(l^*_0,l^*_3)_{\varepsilon}\varepsilon+(l^*_0,l^*_3)_{\varepsilon^2}\varepsilon^2\}\{(l_1,l_2)+(l^*_1,l^*_2)_{\varepsilon}\varepsilon+(l^*_1,l^*_2)_{\varepsilon^2}\varepsilon^2\}}{\{(l_0,l_2)+(l^*_0,l^*_2)_{\varepsilon}\varepsilon+(l^*_0,l^*_3)_{\varepsilon^2}\varepsilon^2\}\{(l_1,l_3)+(l^*_1,l^*_3)_{\varepsilon}\varepsilon+(l^*_1,l^*_3)_{\varepsilon^2}\varepsilon^2\}} \]
simplify and separate the coefficient of $\varepsilon^0$,      $\varepsilon^1$ and $\varepsilon^2$. Coefficients of $\varepsilon^0$ and $\varepsilon^1$ are same as we computed in \eqref{reps0} and \eqref{reps1} respectively, and the coefficient of $\varepsilon^2$ is 
\begin{align}\label{reps2}
r_{\varepsilon^2}(l^*_0,\ldots,l^*_3)=&\frac{\{(l^*_0,l^*_3)(l^*_1,l^*_2)\}_{\varepsilon^2}}{(l_0,l_2)(l_1,l_3)}-r_{\varepsilon}(l^*_0,\ldots,l^*_3)\frac{\{(l^*_0,l^*_2)(l^*_1,l^*_3)\}_{\varepsilon}}{(l_0,l_2)(l_1,l_3)}\notag\\
&-r(_0,\ldots,l_3)\frac{\{(l^*_0,l^*_2)(l^*_1,l^*_3)\}_{\varepsilon^2}}{(l_0,l_2)(l_1,l_3)}
\end{align}
\begin{rem}
The computation of coefficient of $\varepsilon^n$ which is $r_{\varepsilon^n}(l^*_0,\ldots,l^*_3)$ in the truncated polynomial \eqref{trunc_poly} will give us the following:
\begin{align}
\sum_{k=0}^n\left(\left\lbrace\Delta(l^*_0,l^*_2)\Delta(l^*_1,l^*_3)\right\rbrace_{\varepsilon^k}r_{\varepsilon^{n-k}}(l^*_0,\ldots,l^*_3)\right)=\left\lbrace\Delta(l^*_0,l^*_3)\Delta(^*_1,l^*_2)\right\rbrace_{\varepsilon^n},\notag
\end{align}
where $\Delta(l_i,l_j)\neq0$ for $i\neq j$ and $(l^*_0,\ldots,l^*_3)\in C_4\left(\varmathbb{A}^2_{F[\varepsilon]_{n+1}}\right)$
\end{rem}

\subsection{Triple-ratio in $F[\varepsilon]_{\nu}$:}\label{gen_cross}
In this subsection we will discuss triple-ratio (generalized cross-ratio) of 6 points, i.e.,  $(l^*_0,\ldots,l^*_5)\in C_6(\varmathbb{A}^3_{F[\varepsilon]_{\nu}})$ for $\nu=n+1$. We are pleased to see that the calculations in triple-ratio are similar as the cross-ratio of 4 points $(l^*_0,\ldots,l^*_3) \in C_4(\varmathbb{A}^2_{F[\varepsilon]_{\nu}})$. 

\textbf{Case} $\nu=2$:

First we take $(l^*_0,\ldots,l^*_5) \in C_6(\varmathbb{A}^3_{F[\varepsilon]_2})$, for any $l^*_i \in (l^*_0,\ldots,l^*_5)$

\[l^*_i=\left(\begin{array}{c}a_i+a_{i,\varepsilon}\varepsilon\\ b_i+b_{i,\varepsilon}\varepsilon\\ c_i+c_{i,\varepsilon}\varepsilon\end{array}\right)
=\left(\begin{array}{c}a_i\\b_i\\c_i\end{array}\right)+\left(\begin{array}{c}a_{i,\varepsilon}\\b_{i,\varepsilon}\\c_{i,\varepsilon}\end{array}                                                                                                                                   \right)\varepsilon=l_i+l_{i,\varepsilon}\varepsilon\]

\[\Delta(l^*_i,l^*_j,l^*_k)=\Delta(l_i,l_j,l_k)+\Delta(l^*_i,l^*_j,l^*_k)_\varepsilon\varepsilon \]
where $\Delta(l_i,l_j,l_k)$ is a $3\times3$-determinant,
\[\Delta(l^*_i,l^*_j,l^*_k)_\varepsilon=\Delta(l_{i,\varepsilon},l_j,l_k)+\Delta(l_i,l_{j,\varepsilon},l_k)+\Delta(l_i,l_j,l_{k,\varepsilon})\]
and\[\Delta(l^*_i,l^*_j,l^*_k)_{\varepsilon^0}=\Delta(l_{i},l_j,l_k)\]
As we can expand 
\[\textbf{r}_3(l^*_0,\ldots,l^*_5)=r_3(l_0,\ldots,l_5)+r_{3,\varepsilon}(l^*_0,\ldots,l^*_5)\varepsilon\]
From \cite{Gonc1} we have
\[r_3(l_0,\ldots,l_5)=\text{Alt}_6\frac{\Delta(l_0,l_1,l_3)\Delta(l_1,l_2,l_4)\Delta(l_2,l_0,l_5)}{\Delta(l_0,l_1,l_4)\Delta(l_1,l_2,l_5)\Delta(l_2,l_0,l_3)}\]
for $\Delta(l_i,l_j,l_j)\neq0$ multiplicative inverse of $\Delta(l^*_i,l^*_j,l^*_k)$ is $-\frac{1}{\Delta(l_i,l_j,l_j)}-\frac{\Delta(l^*_i,l^*_j,l^*_k)_\varepsilon}{\Delta(l_i,l_j,l_j)^2}\varepsilon$ and from now on we will use $(l^*_il^*_jl^*_k)$ instead of $\Delta(l^*_i,l^*_j,l^*_k)$ unless specify.
\begin{align}
\textbf{r}(l^*_0,\ldots,l^*_5)&=\text{Alt}_6\frac{(l^*_0l^*_1l^*_3)(l^*_1l^*_2l^*_4)(l^*_2l^*_0l^*_5)}{(l^*_0l^*_1l^*_4)(l^*_1l^*_2l^*_5)(l^*_2l^*_0l^*_3)}\notag\\
&=\text{Alt}_6\Bigg\{\frac{\{(l_0l_1l_3)+(l^*_0l^*_1l^*_3)_\varepsilon\varepsilon\}\{(l_1l_2l_4)+(l^*_1l^*_2l^*_4)_\varepsilon\varepsilon\}\{(l_2l_0l_5)+(l^*_2l^*_0l^*_5)_\varepsilon\varepsilon\}}{\{(l_0l_1l_4)+(l^*_0l^*_1l^*_4)_\varepsilon\varepsilon\}\{(l_1l_2l_5)+(l^*_1l^*_2l^*_5)_\varepsilon\varepsilon\}\{(l_2l_0l_3)+(l^*_2l^*_0l^*_3)_\varepsilon\varepsilon\}}\Bigg\}\notag
\end{align}
Simplify the above and separate coefficients of $\varepsilon^0$ and $\varepsilon^1$, we will see that the coefficient of $\varepsilon^1$ is the triple-ratio of six points $(l^*_0,\ldots,l^*_5)\in C_6(\varmathbb{A}^3_F)$ and the coefficient of $\varepsilon$ is the following:
\begin{align}\label{greps1}
&r_{3,\varepsilon}(l^*_0,\ldots,l^*_5)\notag\\
&=\text{Alt}_6\Bigg\{\frac{\{(l^*_0l^*_1l^*_3)(l^*_1l^*_2l^*_4)(l^*_2l^*_0l^*_5)\}_\varepsilon}{(l_0l_1l_4)(l_1l_2l_5)(l_2l_0l_3)}-\frac{(l_0l_1l_3)(l_1l_2l_4)(l_2l_0l_5)}{(l_0l_1l_4)(l_1l_2l_5)(l_2l_0l_3)}\frac{\{(l^*_0l^*_1l^*_4)(l^*_1l^*_2l^*_5)(l^*_2l^*_0l^*_3)\}_\varepsilon}{(l_0l_1l_4)(l_1l_2l_5)(l_2l_0l_3)}\Bigg\}
\end{align}
we can write the explicit formula.
For 6 points in $C_6(\varmathbb{A}^3_{F[\varepsilon]_2})$
\[r_{3,\varepsilon}(l^*_0,\ldots,l^*_5)=\text{Alt}_6\Bigg\{\frac{\{(l^*_0l^*_1l^*_3)(l^*_1l^*_2l^*_4)(l^*_2l^*_0l^*_5)\}_\varepsilon}{(l_0l_1l_4)(l_1l_2l_5)(l_2l_0l_3)}-r_3(l_0,\ldots,l_5)\frac{\{(l^*_0l^*_1l^*_4)(l^*_1l^*_2l^*_5)(l^*_2l^*_0l^*_3)\}_\varepsilon}{(l_0l_1l_4)(l_1l_2l_5)(l_2l_0l_3)}\Bigg\}\]
\textbf{Case} $\nu=3$:

For 6 points in $C_6(\varmathbb{A}^3_{F[\varepsilon]_3})$ we will compute the following through same procedure as we did above, and for $a\neq0$ we have $\frac{1}{a+a_\varepsilon\varepsilon+a_{\varepsilon^2}\varepsilon^2}=b+b_\varepsilon\varepsilon+b_{\varepsilon^2}\varepsilon^2$ where $b=a^{-1}\in F^\times$, $b_\varepsilon=(a,a_\varepsilon)$ and $b_{\varepsilon^2}=(a,a_\varepsilon,a_{\varepsilon^2})$
\begin{align}
r_{3,{\varepsilon^2}}=\text{Alt}_6&\Bigg\{\frac{\{(l^*_0l^*_1l^*_3)(l^*_1l^*_2l^*_4)(l^*_2l^*_0l^*_5)\}_{\varepsilon^2}}{(l_0l_1l_4)(l_1l_2l_5)(l_2l_0l_3)}-r_{3,\varepsilon}(l^*_0,\ldots,l^*_5)\frac{\{(l^*_0l^*_1l^*_4)(l^*_1l^*_2l^*_5)(l^*_2l^*_0l^*_3)\}_\varepsilon}{(l_0l_1l_4)(l_1l_2l_5)(l_2l_0l_3)}\notag\\
-r_3&(l_0,\ldots,l_5)\frac{\{(l^*_1l^*_2l^*_4)(l^*_1l^*_2l^*_5)(l^*_2l^*_0l^*_3)\}_{\varepsilon^2}}{(l_0l_1l_4)(l_1l_2l_5)(l_2l_0l_3)}\Bigg\}\notag
\end{align}
Here we used the following notation just for simplification.
\[(abc)_{\varepsilon}:=a_{\varepsilon}b_{\varepsilon^0}c_{\varepsilon^0}+a_{\varepsilon^0}b_{\varepsilon}c_{\varepsilon^0}+a_{\varepsilon^0}b_{\varepsilon^0}c_{\varepsilon}\]
\section{The Tangent Complex to the Bloch-Suslin Complex}\label{tbsc2}
In this section mainly we will discuss text from \cite{Cath3}. Let $F[\varepsilon]_2=F[\varepsilon]/\varepsilon^2$ be the ring of dual numbers for an arbitrary field $F$. We can define an $F^\times$-action in $F[\varepsilon]_2$ as follows. For $\lambda \in F^\times$,
\[\lambda:F[\varepsilon]_2\rightarrow F[\varepsilon]_2, a+a'\varepsilon\mapsto a+\lambda a'\varepsilon\]
we denote this action by $\star$, so we use $\lambda\star(a+a'\varepsilon)=a+\lambda a'\varepsilon$.
\subsubsection{Definition:}The \emph{tangent group} $T\mathcal{B}_2(F)$ is defined as a $\mathbb{Z}$-module generated by the combinations $[a+a'\varepsilon]-[a]\in\varmathbb{Z}[F[\varepsilon]_2],\quad (a,a'\in F)$: for which we put the shorthand $\langle a;a']:=[a+a'\varepsilon]-[a]$ and quotient by the following relation 
\begin{align}\label{t5term}
\left\langle a;a'\right]-&\left\langle b;b'\right]+\left\langle\frac{b}{a};\left(\frac{b}{a}\right)'\right]-\left\langle\frac{1-b}{1-a};\left(\frac{1-b}{1-a}\right)'\right]\notag\\
+&\left\langle\frac{a(1-b)}{b(1-a)};\left(\frac{a(1-b)}{b(1-a)}\right)'\right],\quad a,b\neq 0,1,a\neq b
\end{align}

where \[\left(\frac{b}{a}\right)'=\frac{ab'-a' b}{a^2},\]
\[\left(\frac{1-b}{1-a}\right)'=\frac{(1-b)a'-(1-a)b'}{(1-a)^2}\]
and
\[\left(\frac{a(1-b)}{b(1-a)}\right)'=\frac{b(1-b)a'-a(1-a)b'}{(b(1-a))^2}\]
\begin{rem}
See \cite{Cath3} for a discussion of $T\mathcal{B}_2(F)$, where the definition of $T\mathcal{B}_2(F)$ was justified using Lemma 3.1 of \cite{Cath3}
\end{rem}

We give a list of relations in $T\mathcal{B}_2(F)$ from \cite{Cath3}. These relations use the $\star$-action in $T\mathcal{B}_2(F)$. By specialization of the five-term relation (\ref{t5term}), we find

1. Two-term relation:
\[\langle a;b]_2=-\langle 1-a;-b]_2\]
2. Inversion relation:
\[\langle a;b]_2=\left\langle \frac{1}{a};-\frac{b}{a^2}\right]_2\]
3. Four-term relation:

If we use $a'=a(1-a)$ and $b'=b(1-b)$ then (\ref{t5term}) becomes four-term relation (see \cite{Cath3}).
\begin{align}
\langle a;a(1-a)]_2-&\langle b;b(1-b)]_2+a\star\left\langle \frac{b}{a};\frac{b}{a}\left(1-\frac{b}{a}\right)\right]_2\notag\\
+&(1-a)\star\left\langle\frac{1-b}{1-a}; \frac{1-b}{1-a}\left(1-\frac{1-b}{1-a}\right)\right]_2=0,\notag
\end{align}
where $a,b\neq0,1, a\neq b$.

The following map is an infinitesimal analogue of $\delta$ (defined in \cite{Gonc1}) and $\partial$ (defined in \cite{Cath3} and \cite{SiDs}), Cathelineau called it \emph{tangential map}.
\[T\mathcal{B}_2(F)\xrightarrow{\partial_\varepsilon}\left(F\otimes F^\times\right) \oplus \left(\bigwedge{}^2F\right)\]
with
\[\partial_\varepsilon\left(\langle a;b]_2\right)=\left(\frac{b}{a}\otimes(1-a)+\frac{b}{1-a}\otimes a\right)+\left(\frac{b}{1-a}\wedge \frac{b}{a}\right)\]
The first term of the complex is in degree one and $\partial_\varepsilon$ has degree +1.

Note that we get the direct sum of two spaces on the right side.

\section{Dilogarithmic Bicomplexes}\label{tbsdb}
In this section we will connect the Grassmannian bicomplex to the Cathelineau's tangential complex in weight 2.

We will use the following notations throughout this section
\[\Delta(l^*_i,l^*_j)_{\varepsilon}=\Delta(l_{i,\varepsilon},l_j)+\Delta(l_i,l_{j,\varepsilon})\quad\quad\quad
\text{and}\quad\quad\quad 
\Delta(l^*_i,l^*_j)_{\varepsilon^0}=\Delta(l_i,l_j)\]
and we will assume that $\Delta(l_i,l_j)\neq0$ (as we often want to divide by such determinants).

Let $C_m(\varmathbb{A}^2_{F[\varepsilon]_2})$ be the free abelian group generated by the configuration of $m$ points in $\varmathbb{A}^2_{F[\varepsilon]_2}$, where $\varmathbb{A}^2_{F[\varepsilon]_2}$ is defined as an affine plane over $F[\varepsilon]_2$. Configurations of $m$ points in $\varmathbb{A}^2_{F[\varepsilon]_2}$ are 2-tuples of vectors over $F[\varepsilon]_2$ modulo $GL_2(F[\varepsilon]_2)$. In this case the Grassmannian complex will be in the following shape
\begin{equation}
\cdots\xrightarrow{d}C_5(\varmathbb{A}^2_{F[\varepsilon]_2})\xrightarrow{d}C_4(\varmathbb{A}^2_{F[\varepsilon]_2})\xrightarrow{d}C_3(\varmathbb{A}^2_{F[\varepsilon]_2}) \notag
\end{equation}
\begin{equation}
d\colon(l^*_0,\ldots,l^*_{m-1})\mapsto\sum_{i=0}^m (-1)^i(l^*_0,\ldots,\hat{l^*_i},\ldots,l^*_{m-1})\notag
\end{equation}
where $l_i^*=\left(\begin{array}{c}a_i+a_{i,\varepsilon}\varepsilon\\b_i+b_{i,\varepsilon}\varepsilon\end{array}\right)=\left(\begin{array}{c}a_i\\ b_i\end{array}\right)+\left(\begin{array}{c}a_{i,\varepsilon}\\b_{i,\varepsilon}\end{array}\right)\varepsilon=l_i+l_{i,\varepsilon}\varepsilon$ and $a_i,b_i,a_{i,\varepsilon},b_{i,\varepsilon}\in F$, $\left(\begin{array}{c}a_i\\ b_i\end{array}\right)\neq \left(\begin{array}{c}0\\ 0\end{array}\right)$


Consider the following diagram
\begin{displaymath}\label{tbicomp2}
 \xymatrix{
C_5(\varmathbb{A}^2_{F[\varepsilon]_2})\ar[r]^{d} &C_4(\varmathbb{A}^2_{F[\varepsilon]_2})\ar[r]^{d}\ar[d]^{\tau_{1,\varepsilon}^2} &C_3(\varmathbb{A}^2_{F[\varepsilon]_2})\ar[d]^{\tau_{0,\varepsilon}^2}\\	
& T\mathcal{B}_2(F)\ar[r]^{\partial_\varepsilon\qquad}	 & F\otimes F^{\times}\oplus \bigwedge^2 F
}\tag{4.2a}
\end{displaymath}
where
\[\partial_\varepsilon:\langle a;b]_2\mapsto\left(\frac{b}{a}\otimes(1-a)+\frac{b}{1-a}\otimes a\right)+\left(\frac{b}{1-a}\wedge\frac{b}{a} \right)\]
We write the map $\tau^2_{0,\varepsilon}$ as a sum of two maps
\[\tau^{(1)}:C_3(\varmathbb{A}^2_{F[\varepsilon]_2})\rightarrow F\otimes F^\times\]
and
\[\tau^{(2)}:C_3(\varmathbb{A}^2_{F[\varepsilon]_2})\rightarrow \bigwedge{}^2F\]
where
\begin{align}\label{t*def1}
&\tau^{(1)}(l^*_0,l^*_1,l^*_2)\notag\\
=&\frac{\Delta(l^*_1,l^*_2)_\varepsilon}{\Delta(l_1,l_2)}\otimes\frac{\Delta(l_0,l_2)}{\Delta(l_0,l_1)}-\frac{\Delta(l^*_0,l^*_2)_\varepsilon}{\Delta(l_0,l_2)}\otimes\frac{\Delta(l_1,l_2)}{\Delta(l_1,l_0)}+\frac{\Delta(l^*_0,l^*_1)_\varepsilon}{\Delta(l_0,l_1)}\otimes\frac{\Delta(l_2,l_1)}{\Delta(l_2,l_0)}\notag
\end{align}
and
\begin{align}
&\tau^{(2)}(l^*_0,l^*_1,l^*_2)\notag\\
=&\frac{\Delta(l^*_0,l^*_1)_\varepsilon}{\Delta(l_0,l_1)}\wedge \frac{\Delta(l^*_1,l^*_2)_\varepsilon}{\Delta(l_1,l_2)}-\frac{\Delta(l^*_0,l^*_1)_\varepsilon}{\Delta(l_0,l_1)}\wedge \frac{\Delta(l^*_0,l^*_2)_\varepsilon}{\Delta(l_0,l_2)}+\frac{\Delta(l^*_1,l^*_2)_\varepsilon}{\Delta(l_1,l_2)}\wedge \frac{\Delta(l^*_0,l^*_2)_\varepsilon}{\Delta(l_0,l_2)}\notag
\end{align}
Furthermore, we put
\[\tau_{1,\varepsilon}^2(l^*_0,\ldots,l^*_3)=\left\langle r(l_0,\ldots,l_3);r_\varepsilon(l^*_0,\ldots,l^*_3)\right]\]
where $r(l_0,\ldots,l_3)$ and $r_\varepsilon(l^*_0,\ldots,l^*_3)$ are the coefficients of $\varepsilon^0$ and $\varepsilon^1$ respectively, in  $\textbf{r}(l^*_0,\ldots,l^*_3)$ as defined in \ref{cr_Feps}.
and $\Delta$ is defined in $\S$\ref{con_w_cr}

Our maps $\tau_{0,\varepsilon}^2$ and $\tau_{1,\varepsilon}^2$ are based on ratios of determinants and cross-ratios respectively, so there is enough evidence that these are independent of the length of the vectors and the volume formed by these vectors. This independence can be seen directly through the definition of the maps.

We will also use the shorthand $(l_il_j)$ instead of $\Delta(l_i,l_j)$ wherever we find less space to accommodate long expressions.

Now calculate 
\begin{align}\label{1mr}
1-\textbf{r}(l^*_0,\ldots,l^*_3)&=\frac{\Delta(l^*_0,l^*_1)\Delta(l^*_2,l^*_3)}{\Delta(l^*_0,l^*_2)\Delta(l^*_1,l^*_3)}\notag\\
                       &=\frac{(l_0l_1)(l_2l_3)}{(l_0l_2)(l_1l_3)}+\frac{y}{(l_0l_2)^2(l_1l_3)^2}\varepsilon 
\end{align}
where
\begin{align}
y=&+(l_0l_2)(l_1l_3)(l_0l_1)(l_2l_{3,\varepsilon})+(l_0l_2)(l_1l_3)(l_0l_1)(l_{2,\varepsilon}l_3)\notag\\
&+(l_0l_2)(l_1l_3)(l_2l_3)(l_0l_{1,\varepsilon})+(l_0l_2)(l_1l_3)(l_2l_3)(l_{0,\varepsilon}l_1)\notag\\
&-(l_0l_1)(l_2l_3)(l_0l_2)(l_1l_{3,\varepsilon})-(l_0l_1)(l_2l_3)(l_0l_2)(l_{1,\varepsilon}l_3)\notag\\
&-(l_0l_1)(l_2l_3)(l_1l_3)(l_0l_{2,\varepsilon})-(l_0l_1)(l_2l_3)(l_1l_3)(l_{0,\varepsilon}l_2)\notag
\end{align}
\begin{rem}
The $F^\times$-action of $T\mathcal{B}_2(F)$ lifts to an $F^\times$-action on  $C_4(\varmathbb{A}^2_{F[\varepsilon]_2})$ in the obvious way: 
\end{rem}
The $F^\times$-action is defined above for $F[\varepsilon]_2$ induces an  $F^\times$-action in
$\varmathbb{A}^2_{F[\varepsilon]_2}$ diagonally as 
\[\lambda\star\left(\begin{array}{c}
                    a+a_{\varepsilon}\varepsilon\\b+b_{\varepsilon}\varepsilon\end{array}
                                                                             \right)=\left(\begin{array}{c}
                    a+\lambda a_{\varepsilon}\varepsilon\\b+\lambda b_{\varepsilon}\varepsilon\end{array}
                                                                             \right)\in \varmathbb{A}^2_{F[\varepsilon]_2},\lambda\in F^\times                  
\]
\begin{lem}\label{claim2}
The diagram (\ref{tbicomp2}) is commutative
\end{lem}
\begin{proof}
This requires direct calculation.
\end{proof}

In the remainder of this section we prove that the following diagram following is a bicomplex.
\begin{displaymath}\label{tbicomp2i}
 \xymatrix{
C_5(\varmathbb{A}^3_{F[\varepsilon]_2})\ar[r]^{d}\ar[d]^{d'}&C_4(\varmathbb{A}^3_{F[\varepsilon]_2})\ar[d]^{d'}\\
C_4(\varmathbb{A}^2_{F[\varepsilon]_2})\ar[r]^{d}\ar[d]^{\tau_{1,\varepsilon}^2} &C_3(\varmathbb{A}^2_{F[\varepsilon]_2})\ar[d]^{\tau_{0,\varepsilon}^2}\\	
T\mathcal{B}_2(F)\ar[r]^{\partial_\varepsilon\qquad}	 & F\otimes F^{\times}\oplus \bigwedge{}^2 F
}\tag{4.2b}
\end{displaymath}
To prove the above is a bicomplex, we are giving the following results.
\begin{prop}\label{zero2}
The map $C_4(\varmathbb{A}^3_{F[\varepsilon]_2})\xrightarrow{d'}C_3(\varmathbb{A}^2_{F[\varepsilon]_2})\xrightarrow{\tau^2_{0,\varepsilon}}\left(F\otimes F^\times\right) \oplus \left(\bigwedge{}^2F\right)$ is zero.
\end{prop}
\begin{proof}
Let $\omega\in\det V^*_3$ be the volume form in three-dimensional vector space $V_3$, i.e., $\Delta(l_i,l_j,l_k)=\langle \omega,l_i\wedge l_j\wedge l_k\rangle$ then $\Delta(l_i,\cdot,\cdot)$ is a volume form in $V_3/\langle l_i\rangle$. Use 
\[\Delta(l^*_i,l^*_j,l^*_k)=\Delta(l_i,l_j,l_k)+\left\lbrace\Delta(l^*_i,l^*_j,l^*_k)_\varepsilon\right\rbrace\varepsilon\]
where
\[\Delta(l^*_i,l^*_j,l^*_k)_\varepsilon=\Delta(l_{i,\varepsilon},l_j,l_k)+\Delta(l_i,l_{j,\varepsilon},l_k)+\Delta(l_i,l_j,l_{k,\varepsilon})\]
We can directly compute $\tau^2_{0,\varepsilon}\circ d'$ which gives zero.
\end{proof}

The following result is very important for proving Theorem \ref{claim4b}. Through this result we are able to see the projected-five term relation in $T\mathcal{B}_2(F)$.
\begin{lem}\label{5pt}
Let $x^*_0,\ldots,x^*_4\in \varmathbb{P}^2_{F[\varepsilon]_2}$ be 5 points in generic position, then
\begin{align}\label{5termeps}
\sum_{i=0}^4(-1)^i\left\langle r(x_i|x_0,\ldots,\hat{x}_i,\ldots,x_4);r_\varepsilon(x^*_i|x^*_0,\ldots,\hat{x}^*_i,\ldots,x^*_4)\right]=0\in T\mathcal{B}_2(F),
\end{align}
where $x^*_i=x_i+x'_i\varepsilon$ and $x_i,x'_i\in\varmathbb{P}^2_F$
\[\textbf{r}(x^*_i|x^*_0,\ldots,\hat{x}^*_i,\ldots,x^*_4)=r(x_i|x_0,\ldots,\hat{x}_i,\ldots,x_4)+r_\varepsilon(x^*_i|x^*_0,\ldots,\hat{x}^*_i,\ldots,x^*_4)\varepsilon,\]
where the LHS denotes the projected cross-ratio of any four points projected from the fifth from $x^*_0,\ldots,x^*_4\in\varmathbb{P}^2_{F[\varepsilon]_2}$. 
\end{lem}
\begin{proof}
Consider five points $y_0,\ldots,y_4\in \varmathbb{P}^1_F$ in generic position. We can write the five-term relation in terms of cross-ratios in $\mathcal{B}_2(F)$ as (see Proposition 4.5 (2)b in \cite{PandG}):
\[\sum_{i=0}^4(-1)^i[r(y_0,\ldots,\hat{y}_i,\ldots,y_4)]_2=0\]
These five points depend on 2 parameters modulo the action of $PGL_2(F)$, whose action on $\varmathbb{P}^1_F$ is 3-fold transitive,  so we can express these five points with two variables modulo this action, we can put
\begin{align*}
\left(y_0,\ldots,y_4\right)=\left(\left(\begin{array}{c}1\\0\end{array}\right),\left(\begin{array}{c}0\\1\end{array}\right),\left(\begin{array}{c}1\\1\end{array}\right),\left(\begin{array}{c}\frac{1}{a}\\1\end{array}\right),\left(\begin{array}{c}\frac{1}{b}\\1\end{array}\right)\right),
\end{align*}
then we get one of the form of five-term relation in two variables (needs to use inversion in the last two terms).
\[[a]_2-[b]_2+\left[\frac{b}{a}\right]_2+\left[\frac{1-a}{1-b}\right]_2-\left[\frac{1-\frac{1}{a}}{1-\frac{1}{b}}\right]_2=0.\]
Now we consider five points $y^*_0,\ldots,y^*_4\in\varmathbb{P}^1_{F[\varepsilon]_2}$, in generic position, where $y^*_i=y_i+y'_{i}\varepsilon$ for $y_i,y'_{i}\in\varmathbb{P}^1_F$. A generic $2\times2$ matrix in $PGL_2(F[\varepsilon]_2)$ depends on $6=2(2\times2)-2(1)$ parameters, while each point in $\varmathbb{P}^1_{F[\varepsilon]_2}$ depends on 2 parameters, so these five points in $\varmathbb{P}^1_{F[\varepsilon]_2}$ modulo the action of $PGL_2(F[\varepsilon]_2)$ have 4 parameters. Now we can express them by using four variables we choose:
\begin{align*}
\left(y^*_0,\ldots,y^*_4\right)=\left(\left(\begin{array}{c}1\\0\end{array}\right),\left(\begin{array}{c}0\\1\end{array}\right),\left(\begin{array}{c}1\\1\end{array}\right),\left(\begin{array}{c}\frac{1}{a}-\frac{a'}{a^2}\varepsilon\\1\end{array}\right),\left(\begin{array}{c}\frac{1}{b}-\frac{b'}{b^2}\varepsilon\\1\end{array}\right)\right).
\end{align*}
We calculate all possible determinants which are the following:
\begin{align*}
\Delta(y_0,y_1)=\Delta(y_0,y_2)=\Delta(y_0,y_3)=\Delta(y_0,y_4)=1,\Delta(y_1,y_2)=-1,\\
\Delta(y_1,y_3)=-\frac{1}{a},\Delta(y_1,y_4)=-\frac{1}{b},\Delta(y_2,y_3)=1-\frac{1}{a},\Delta(y_2,y_4)=1-\frac{1}{b}\\
\Delta(y^*_0,y^*_1)_\varepsilon=\Delta(y^*_0,y^*_2)_\varepsilon=\Delta(y^*_0,y^*_3)_\varepsilon=\Delta(y^*_0,y^*_4)_\varepsilon=\Delta(y^*_1,y^*_2)_\varepsilon=0\\
\Delta(y^*_1,y^*_3)_\varepsilon=\Delta(y^*_2,y^*_3)_\varepsilon=\frac{a'}{a^2},\Delta(y^*_1,y^*_4)_\varepsilon=\Delta(y^*_2,y^*_4)_\varepsilon=\frac{b'}{b^2}
\end{align*}
For $y^*_0,\ldots,y^*_4\in\varmathbb{P}^1_{F[\varepsilon]_2}$, we can write the following expression in $T\mathcal{B}_2(F)$
\[\sum^4_{i=0}(-1)^i\left\langle r(y_0,\ldots,\hat{y}_i,\ldots,y_4);r_\varepsilon(y^*_0,\ldots,\hat{y}^*_i,\ldots,y^*_4)\right]_2\]
If we expand the above expression and we put all determinants in it we will get the following expression in two variables.
\begin{align}
\left\langle a;a'\right]_2-&\left\langle b;b'\right]_2+\left\langle\frac{b}{a};\frac{ab'-a' b}{a^2}\right]_2-\left\langle\frac{1-b}{1-a};\frac{(1-b)a'-(1-a)b'}{(1-a)^2}\right]_2\notag\\
+&\left\langle\frac{a(1-b)}{b(1-a)};\frac{b(1-b)a'-a(1-a)b'}{(b(1-a))^2}\right]_2\notag
\end{align}
From \eqref{5termeps} it is clear that the above is the LHS of the five-term relation in $T\mathcal{B}_2(F)$. We will reduce the claim to this latter form of five-term relation.

Consider $x_0,\ldots,x_4\in\varmathbb{P}^2_F$ in generic position. These five points also depend on 2 parameters modulo the action of $PGL_2(F)$, so we can express these five points in terms of two variables by the following choice:
\[\left(x_0,\ldots,x_4\right)=\left(\left(\begin{array}{c}1\\0\\0\end{array}\right),\left(\begin{array}{c}0\\1\\0\end{array}\right),\left(\begin{array}{c}0\\0\\1\end{array}\right),\left(\begin{array}{c}1\\1\\1\end{array}\right),\left(\begin{array}{c}\frac{1}{b}\\ \frac{1}{a}\\1\end{array}\right)\right)\]
We compute all possible $3\times3$ determinants of the above and put them in the expansion of the following:
\[\sum^4_{i=0}(-1)^i\left[r(x_i|x_0,\ldots,\hat{x}_i,\ldots,x_4)\right]_2\in \mathcal{B}_2(F),\]
we get the following expression in two variables
\[[a]_2-b]_2+\left[\frac{b}{a}\right]_2+\left[\frac{1-a}{1-b}\right]_2-\left[\frac{1-\frac{1}{a}}{1-\frac{1}{b}}\right]_2,\]
clearly the above is the LHS of one version of five-term relation in $\mathcal{B}_2(F)$.

Since by assumption $x^*_0,\ldots,x^*_4\in \varmathbb{P}^2_{F[\varepsilon]_2}$ are 5 points in generic position, we can express them modulo the action of $PGL_3(F[\varepsilon]_2)$ into 4 parameters then we can choose these points in terms of four variables in the following way:
\[\left(x^*_0,\ldots,x^*_4\right)=\left(\left(\begin{array}{c}1\\0\\0\end{array}\right),\left(\begin{array}{c}0\\1\\0\end{array}\right),\left(\begin{array}{c}0\\0\\1\end{array}\right),\left(\begin{array}{c}1\\1\\1\end{array}\right),\left(\begin{array}{c}\frac{1}{b}-\frac{b'}{b^2}\varepsilon\\ \frac{1}{a}-\frac{a'}{a^2}\varepsilon\\1\end{array}\right)\right)\]
We compute all possible $3\times3$ determinants and substitute them in an expansion of the following:
\[\sum_{i=0}^4(-1)^i\left\langle r(x_i|x_0,\ldots,\hat{x}_i,\ldots,x_4);r_\varepsilon(x^*_i|x^*_0,\ldots,\hat{x}^*_i,\ldots,x^*_4)\right]_2\in T\mathcal{B}_2(F),\]
we get
\begin{align}
\left\langle a;a'\right]_2-&\left\langle b;b'\right]_2+\left\langle\frac{b}{a};\frac{ab'-a' b}{a^2}\right]_2-\left\langle\frac{1-b}{1-a};\frac{(1-b)a'-(1-a)b'}{(1-a)^2}\right]_2\notag\\
+&\left\langle\frac{a(1-b)}{b(1-a)};\frac{b(1-b)a'-a(1-a)b'}{(b(1-a))^2}\right]_2\notag
\end{align}
which is the five-term expression in $T\mathcal{B}_2(F)$ up to invoking the inversion relation for the last two terms, which also holds in $T\mathcal{B}_2(F)$ 
\end{proof}

Lemma \ref{5pt} indicates that we now have the projected five-term relation in $T\mathcal{B}_2(F)$ and this relation will help us to prove the commutative diagram for weight $n=3$ in the tangential case.

\begin{prop}\label{one2}
The map $C_5(\varmathbb{A}^3_{F[\varepsilon]_2})\xrightarrow{d'}C_4(\varmathbb{A}^2_{F[\varepsilon]_2})\xrightarrow{\tau^2_{1,\varepsilon}}T\mathcal{B}_2(F)$ is zero.
\end{prop}
\begin{proof}
We can directly calculate $\tau^2_{1\varepsilon}\circ d'$.
\begin{align}
\tau^2_{1,\varepsilon}\circ d'(l^*_0,\ldots,l^*_4)=\tau^2_{1,\varepsilon}\left(\sum^4_{i=0}(-1)^i\left(l^*_i|l^*_0,\ldots,\hat{l}^*_i,\ldots,l^*_4\right)\right)\notag \\
=\sum^4_{i=0}(-1)^i\left\langle r\left(l_i|l_0,\ldots,\hat{l}_i\ldots,l_4\right);r_\varepsilon\left(l^*_i|l^*_0,\ldots,\hat{l}^*_i,\ldots,l^*_4\right)\right]_2
\end{align}
The above is the projected five term relation in $T\mathcal{B}_2(F)$ by Lemma \ref{5pt}.
\end{proof}

Theorem \ref{claim2} shows that the diagram \eqref{tbicomp2} is commutative and Propositions \ref{zero2} and \ref{one2} shows that we have formed a bicomplex between the Grassmannian complex and  Cathelineau's tangential complex.

\section{Trilogarithmic Complexes}\label{tbstb}
We have already discussed the tangent group (or $\varmathbb{Z}$-module) $T\mathcal{B}_2(F)$ over $F[\varepsilon]_2$ in $\S$\ref{tbsdb}. In this section we will discuss group $T\mathcal{B}_3(F)$ and its functional equations and will connect Grassmannian complex and tangential complex to Goncharov complex.

\subsection{Definition and functional equations of $T\mathcal{B}_3(F)$:}
The $\varmathbb{Z}$-module $T\mathcal{B}_3(F)$ over $F[\varepsilon]_2$ is defined as the group generated by:
\[\langle a;b]=[a+b\varepsilon]-[a]\in\varmathbb{Z}\left[F[\varepsilon]_2\right], \quad a,b \in F,\quad a\neq0,1\]
and quotiented by the kernel of the following map
\[\partial_{\varepsilon,3}:\varmathbb{Z}\left[F[\varepsilon]_2\right]\rightarrow T\mathcal{B}_2(F)\otimes F^\times\oplus F\otimes \mathcal{B}_2(F),\langle a;b]\mapsto \langle a;b]_2\otimes a+\frac{b}{a}\otimes [a]_2\]
Now we can say that $\langle a;b]_3\in T\mathcal{B}_3(F)\subset\varmathbb{Z}[F[\varepsilon]_2]/\ker\partial_{\varepsilon,3}$

We have the following relations which are satisfied in $T\mathcal{B}_3(F)$.
\begin{enumerate}
\item The three-term relation.
\[\langle 1-a;(1-a)_\varepsilon]_3-\langle a;a_\varepsilon]_3-\left\langle 1-\frac{1}{a};\left(1-\frac{1}{a}\right)_\varepsilon\right]_3=0 \in T\mathcal{B}_3(F)\]

\item The inversion relation
\[\langle a;a_\varepsilon]_3=\left\langle \frac{1}{a};\left(\frac{1}{a}\right)_\varepsilon\right]_3\]

\item The Cathelineau 22-term relation (\cite{PandG})

This relation $J(a,b,c)$ for the indeterminates $a,b,c$ can be written in this way:
\begin{equation}\label{22base}
J(a,b,c)=\left[\left[a,c\right]\right]-\left[\left[b,c\right]\right]+a\left[\left[\frac{b}{a},c\right]\right]+(1-a)\left[\left[\frac{1-b}{1-a},c\right]\right],
\end{equation}
where
\[\left[\left[a,b\right]\right]=(b-a)\tau(a,b)+\frac{1-b}{1-a}\sigma(a)+\frac{1-a}{1-b}\sigma(b),\]
while $\tau(a,b)$ is defined via five term relation and $\star$-action. We take $\langle x_i;x_{i,\varepsilon}]_3$ with coefficient $\frac{1}{1-x_i}$ which is handled by $\star$-action. 
\begin{align*}
\tau(a,b)=&\left\langle a;a_\varepsilon\cdot\frac{1}{1-a}\right]_3-\left\langle b;b_\varepsilon\cdot\frac{1}{1-b}\right]_3+\left\langle \frac{b}{a};\left(\frac{b}{a}\right)_\varepsilon\cdot\frac{1}{a-b}\right]_3\\
-&\left\langle \frac{1-b}{1-a};\left(\frac{1-b}{1-a}\right)_\varepsilon\cdot\frac{1}{b-a}\right]_3-\left\langle \frac{a(1-b)}{b(1-a)};\left(\frac{a(1-b)}{b(1-a)}\right)_\varepsilon\cdot\frac{1}{b-a}\right]_3
\end{align*}
and
\[\sigma(a)=\langle a;a_\varepsilon\cdot a]_3+\langle 1-a;(1-a)_\varepsilon\cdot (1-a)]_3.\]
Then we can calculate Cathelineau's 22-term expression by substituting all values in (\ref{22base}).
\begin{align}\label{22eps}
J(a,b,c)=&\langle a;a_\varepsilon c]_3 -\langle b;b_\varepsilon c]_3+\langle c;c_\varepsilon(a-b+1)]_3\notag\\
+&\langle 1-a;(1-a)_\varepsilon(1-c)]_3-\langle1-b;(1-b)_\varepsilon(1-c)]_3\notag\\
+&\langle1-c;(1-c)_\varepsilon(b-a)]_3-\left\langle \frac{c}{a};\left(\frac{c}{a}\right)_\varepsilon\right]_3+\left\langle\frac{c}{b};\left(\frac{c}{b}\right)_\varepsilon\right]_3+\left\langle \frac{b}{a};\left(\frac{b}{a}\right)_\varepsilon c\right]_3\notag\\
-&\left\langle \frac{1-c}{1-a};\left(\frac{1-c}{1-a}\right)_\varepsilon\right]_3+\left\langle \frac{1-c}{1-b};\left(\frac{1-c}{1-b}\right)_\varepsilon\right]_3+\left\langle \frac{1-b}{1-a};\left(\frac{1-b}{1-a}\right)_\varepsilon c\right]_3\notag\\
+&\left\langle\frac{a(1-c)}{c(1-a)};\left(\frac{a(1-c)}{c(1-a)}\right)_\varepsilon\right]_3-\left\langle\frac{ca}{b};\left(\frac{ca}{b}\right)_\varepsilon\right]_3-\left\langle \frac{b(1-c)}{c(1-b)};\left(\frac{b(1-c)}{c(1-b)}\right)_\varepsilon\right]_3\notag\\
+&\left\langle \frac{a-b}{a};\left(\frac{a-b}{a}\right)_\varepsilon(1-c)\right]_3+\left\langle\frac{b-a}{1-a};\left(\frac{b-a}{1-a}\right)_\varepsilon(1-c)\right]_3\notag\\
+&\left\langle \frac{c(1-a)}{1-b};\left(\frac{c(1-a)}{1-b}\right)_\varepsilon\right]_3-\left\langle \frac{(1-c)a}{a-b};\left(\frac{(1-c)a}{a-b}\right)_\varepsilon\right]_3\notag\\
-&\left\langle \frac{(1-c)(1-a)}{b-a};\left(\frac{(1-c)(1-a)}{b-a}\right)_\varepsilon\right]_3\notag\\
+&\left\langle\frac{(1-c)b}{c(a-b)};\left(\frac{(1-c)b}{c(a-b)}\right)_\varepsilon\right]_3+\left\langle\frac{(1-c)(1-b)}{c(b-a)};\left(\frac{(1-c)(1-b)}{c(b-a)}\right)_\varepsilon\right]_3
\end{align}
\end{enumerate}
For the special condition $a_\varepsilon=a(1-a)$,$b_\varepsilon=b(1-b)$ and $c_\varepsilon=c(1-c)$, this 22-term expression becomes zero in $T\mathcal{B}_3(F)$.

One can write the following complex for $T\mathcal{B}_3(F)$.
\begin{displaymath}
T\mathcal{B}_3(F)\xrightarrow{\partial_\varepsilon}\substack{T\mathcal{B}_2(F)\otimes F^\times\\ \oplus\\ F\otimes \mathcal{B}_2(F)}\xrightarrow{\partial_\varepsilon}\left(F\otimes\bigwedge{}^2F^\times\right)\oplus\left(\bigwedge{}^3F\right)
\end{displaymath}
\subsection{Mapping Grassmannian complexes to Tangential complexes in weight 3:}
In this subsection, we will try to find morphisms between this complex and the Grassmannian complex and after a long computation we see that each square of the following diagram is commutative. Consider the following diagram
\begin{displaymath}\label{bicomp3b}
\xymatrix{
C_6(\varmathbb{A}^3_{F[\varepsilon]_2})\ar[r]^{d}\ar[d]^{\tau^3_{2,\varepsilon}}     & C_5(\varmathbb{A}^3_{F[\varepsilon]_2})\ar[r]^{d}\ar[d]^{\tau_{1,\varepsilon}^{3}}        &C_4(\varmathbb{A}^3_{F[\varepsilon]_2})\ar[d]^{\tau_{0,\varepsilon}^{3}}\\
T\mathcal{B}_3(F)\ar[r]^{\partial_\varepsilon\qquad\qquad}     & \left(T\mathcal{B}_2(F)\otimes F^\times\right) \oplus  \left(F\otimes \mathcal{B}_2(F)\right)\ar[r]^{\qquad\partial_\varepsilon}          & \left(F\otimes \bigwedge^2 F^\times\right)\oplus \left(\bigwedge^3F\right)}\ \tag{4.3a}
\end{displaymath}
Here we define
\[\textbf{r}(l^*_0|l^*_1,l^*_2,l^*_3,l^*_4)=\frac{\Delta(l^*_0,l^*_1,l^*_4)\Delta(l^*_0,l^*_2,l^*_3)}{\Delta(l^*_0,l^*_1,l^*_3)\Delta(l^*_0,l^*_2,l^*_4)}\]
The projected cross-ratio is defined here 
\[\textbf{r}(l^*_0|l^*_1,l^*_2,l^*_3,l^*_4)=r(l_0|l_1,l_2,l_3,l_4)+r_\varepsilon(l^*_0|l^*_1,l^*_2,l^*_3,l^*_4)\varepsilon 
\]
where
\[r(l_0|l_1,l_2,l_3,l_4)=\frac{\Delta(l_0,l_1,l_4)\Delta(l_0,l_2,l_3)}{\Delta(l_0,l_1,l_3)\Delta(l_0,l_2,l_4)}\]
\[r_\varepsilon(l^*_0|l^*_1,l^*_2,l^*_3,l^*_4)=\frac{u}{\Delta(l_0,l_1,l_3)^2\Delta(l_0,l_2,l_4)^2}\]
\begin{align}
u=&-\Delta(l_0,l_1,l_4)\Delta(l_0,l_2,l_3)\{\Delta(l_0,l_1,l_3)\Delta(l^*_0,l^*_2,l^*_4)_\varepsilon+\Delta(l_0,l_2,l_4)\Delta(l^*_0,l^*_1,l^*_3)_\varepsilon\}\notag\\
&+\Delta(l_0,l_1,l_3)\Delta(l_0,l_2,l_4)\{\Delta(l_0,l_1,l_4)\Delta(l^*_0,l^*_2,l^*_3)_\varepsilon+\Delta(l_0,l_2,l_3)\Delta(l^*_0,l^*_1,l^*_4)_\varepsilon\}\notag
\end{align}
where the morphisms between the two complexes are defined as follows:
\begin{align}
&\tau^3_{0,\varepsilon}(l^*_0,\ldots,l^*_3)\notag\\
=&\sum_{i=0}^3(-1)^{i}\Bigg(\frac{\Delta(l^*_0,\ldots,\hat{l}^*_i,\ldots,l^*_3)_\varepsilon}{\Delta(l_0,\ldots,\hat{l}_i,\ldots,l_3)}\otimes\frac{\Delta(l_0,\ldots,\hat{l}_{i+1},\ldots,l_3)}{\Delta(l_0,\ldots,\hat{l}_{i+2},\ldots,l_3)}\notag\\
&\quad\quad\quad\wedge\frac{\Delta(l_0,\ldots,\hat{l}_{i+3},\ldots,l_3)}{\Delta(l_0,\ldots,\hat{l}_{i+2},\ldots,l_3)}+
\bigwedge_{\substack{j=0\\j\neq i}}^3\frac{\Delta(l^*_0,\ldots,\hat{l}^*_j,\ldots,l^*_3)_\varepsilon}{\Delta(l_0,\ldots,\hat{l}_j,\ldots,l_3)}\Bigg),\quad i\mod 4,\notag
\end{align}
\begin{align}
\tau^3_{1,\varepsilon}(l^*_0,\ldots,l^*_4)\notag\\ =-\frac{1}{3}\sum_{i=0}^4(-1)^i\Bigg(&\left\langle r(l_i|l_0,\ldots,\hat{l}_i,\ldots,l_4);r_\varepsilon(l^*_i|l^*_0,\ldots,\hat{l^*}_i,\ldots,l^*_4)\right]_2\otimes \prod_{i\neq j}\Delta(\hat{l}_i,\hat{l}_j)\notag\\
&+\sum_{\substack{j=0 \\ j\neq i}}^4\left(\frac{\Delta(l^*_0,\ldots,\hat{l}^*_i,\ldots,\hat{l}^*_j,\ldots,l^*_4)_\varepsilon}{\Delta(l_0,\ldots,\hat{l}_i,\ldots,\hat{l}_j,\ldots,l_4)}\right)\otimes\left[r(l_i|l_0,\ldots,\hat{l}_i,\ldots,l_4)\right]_2\Bigg)\notag
\end{align}
and
\begin{align}
\tau^3_{2\varepsilon}(l^*_0,\ldots,l^*_5)=\frac{2}{45}\text{Alt}_6\left\langle r_3(l_0,\ldots,l_5);r_{3,\varepsilon}(l^*_0,\ldots,l^*_5)\right]_3\notag
\end{align}
where
\[r_3(l_0,\ldots,l_5)=\frac{(l_0l_1l_3)(l_1l_2l_4)(l_2l_0l_5)}{(l_0l_1l_4)(l_1l_2l_5)(l_2l_0l_3)}\]
and
\begin{align}
r_{3,\varepsilon}&(l^*_0,\ldots,l^*_5)\notag\\
=&\frac{\{(l^*_0l^*_1l^*_3)(l^*_1l^*_2l^*_4)(l^*_2l^*_0l^*_5)\}_\varepsilon}{(l_0l_1l_4)(l_1l_2l_5)(l_2l_0l_3)}-\frac{(l_0l_1l_3)(l_1l_2l_4)(l_2l_0l_5)}{(l_0l_1l_4)(l_1l_2l_5)(l_2l_0l_3)}\frac{\{(l^*_0l^*_1l^*_4)(l^*_1l^*_2l^*_5)(l^*_2l^*_0l^*_3)\}_\varepsilon}{(l_0l_1l_4)(l_1l_2l_5)(l_2l_0l_3)}
\end{align}
the map $\partial_\varepsilon$ is defined as
\begin{align}
\partial_\varepsilon&\left(\langle a;b]_2\otimes c+x\otimes[y]_2\right)\notag\\
=&\left(-\frac{b}{1-a}\otimes a\wedge c-\frac{b}{a}\otimes(1-a)\wedge c+x\otimes(1-y)\wedge y\right)+\left(\frac{b}{1-a}\wedge\frac{b}{a}\wedge x\right)\notag
\end{align}
and
\[\partial_\varepsilon(\langle a;b]_3)= \langle a;b]_2\otimes a+\frac{b}{a}\otimes [a]_2 \]

\begin{thm}\label{claim4a}
The right square of the diagram (\ref{bicomp3b}), i.e.
\begin{displaymath}\label{claim4adia}
\xymatrix{
C_5(\varmathbb{A}^3_{F[\varepsilon]_2})\ar[d]^{\tau^3_{1,\varepsilon}}\ar[rr]^{d}&&C_4(\varmathbb{A}^3_{F[\varepsilon]_2})\ar[d]^{\tau^3_{0,\varepsilon}}\\
\left(T\mathcal{B}_2(F)\otimes F^\times\right)\oplus\left(F\otimes \mathcal{B}_2(F)\right)\ar[rr]^{\partial_\varepsilon}&&\left(F\otimes\bigwedge^2F^\times\right)\oplus\left(\bigwedge^3F\right)}
\end{displaymath}

is commutative, i.e. $\tau^3_{0,\varepsilon}\circ d=\partial_\varepsilon\circ\tau^3_{1,\varepsilon}$
\end{thm}
\begin{proof}
First we divide the map $\tau^3_{0,\varepsilon}=\tau^{(1)}+\tau^{(2)}$ then calculate $\tau^{(1)}\circ d(l^*_0,\ldots,l^*_4)$
\begin{align}
\tau^{(1)}\circ d(l^*_0,\ldots,l^*_4)
=\tau^3_{0,\varepsilon}\left(\sum^4_{i=0}(-1)^i(l^*_0,\ldots,\hat{l}^*_i,\ldots,l^*_4)\right)\notag
\end{align}
\begin{align}\label{first0q}
=\widetilde{\text{Alt}}_{(01234)}\Bigg(&\sum_{i=0}^3(-1)^{i}\Big(\frac{\Delta(l^*_0,\ldots,\hat{l}^*_i,\ldots,l^*_3)_\varepsilon}{\Delta(l_0,\ldots,\hat{l}_i,\ldots,l_3)}\otimes\frac{\Delta(l_0,\ldots,\hat{l}_{i+1},\ldots,l_3)}{\Delta(l_0,\ldots,\hat{l}_{i+2},\ldots,l_3)}\notag\\
&\quad\quad\quad\wedge\frac{\Delta(l_0,\ldots,\hat{l}_{i+3},\ldots,l_3)}{\Delta(l_0,\ldots,\hat{l}_{i+2},\ldots,l_3)}
\Big),\quad i\mod 4\Bigg)
\end{align}
We expand the inner sum first that contains 12 terms and passing alternation to the sum, gives us 60 different terms overall. We collect terms involving same $\frac{\Delta(l^*_i,l^*_j,l^*_k)}{\Delta(l_i,l_j,l_k)}\otimes\cdots$ together for calculation purpose. On the other hand second part of the map is the following:
\begin{align}\label{sec0q}
&\tau^{(1)}\circ d(l^*_0,\ldots,l^*_4)\notag\\
=&\widetilde{\text{Alt}}_{(01234)}\Bigg(\sum^3_{i=0}(-1)^i\bigwedge_{\substack{j=0\\j\neq i}}^3\frac{\Delta(l^*_0,\ldots,\hat{l}^*_j,\ldots,l^*_3)_\varepsilon}{\Delta(l_0,\ldots,\hat{l}_j,\ldots,l_3)}\Bigg)
\end{align}
The other side of the proof requires very long computations. For the calculation of $\partial_\varepsilon\circ\tau^3_{1,\varepsilon}$ we will use short hand $(l^*_il^*_jl^*_k)_\varepsilon$ for $\Delta(l^*_i,l^*_j,l^*_k)_\varepsilon$ and $(l_il_jl_k)$ for $\Delta((l_i,l_j,l_k)$.  First we write  $\partial_\varepsilon\circ\tau^3_{1,\varepsilon}(l^*_0,\ldots,l^*_4)$ by using the definitions above.
\begin{align}
=\partial_\varepsilon\Bigg(-\frac{1}{3}\sum_{i=0}^4(-1)^i\Big(&\left\langle r(l_i|l_0,\ldots,\hat{l}_i,\ldots,l_4);r_\varepsilon(l^*_i|l^*_0,\ldots,\hat{l^*}_i,\ldots,l^*_4)\right]_2\otimes \prod_{i\neq j}\Delta(\hat{l}_i,\hat{l}_j)\notag\\
&+\sum_{\substack{j=0 \\ j\neq i}}^4\left(\frac{\Delta(l^*_0,\ldots,\hat{l}^*_i,\ldots,\hat{l}^*_j,\ldots,l^*_4)_\varepsilon}{\Delta(l_0,\ldots,\hat{l}_i,\ldots,\hat{l}_j,\ldots,l_4)}\right)\otimes\left[r(l_i|l_0,\ldots,\hat{l}_i,\ldots,l_4)\right]_2\Big)\Bigg)\notag
\end{align}
then we divide $\partial_\varepsilon=\partial^{(1)}+\partial^{(2)}$. The first part $\partial^{(1)}\circ\tau^3_{1,\varepsilon}(l^*_0,\ldots,l^*_4)$ is
\begin{align}\label{first1q}
=-\frac{1}{3}\sum^4_{i=0}(-1)^i\Bigg(-&\frac{r_\varepsilon(l^*_i|l^*_0,\ldots,\hat{l}^*_i,\ldots,l^*_4)}{1-r(l_i|l_0,\ldots,\hat{l}_i,\ldots,l_4)}\otimes r(l_i|l_0,\ldots,\hat{l}_i,\ldots,l_4)\wedge \prod_{i\neq j}(\hat{l}_i,\hat{l}_j)\notag\\
-&\frac{r_\varepsilon(l^*_i|l^*_0,\ldots,\hat{l}^*_i,\ldots,l^*_4)}{r(l_i|l_0,\ldots,\hat{l}_i,\ldots,l_4)}\otimes (1-r(l_i|l_0,\ldots,\hat{l}_i,\ldots,l_4))\wedge\prod_{i\neq j}(\hat{l}_i,\hat{l}_j)\notag\\
+&\sum_{\substack{j=0 \\ j\neq i}}^4\left(\frac{\Delta(l^*_0,\ldots,\hat{l}^*_i,\ldots,\hat{l}^*_j,\ldots,l^*_4)_\varepsilon}{\Delta(l_0,\ldots,\hat{l}_i,\ldots,\hat{l}_j,\ldots,l_4)}\right)\otimes(1-r(l_i|l_0,\ldots,\hat{l}_i,\ldots,l_4))\notag\\
\wedge& r(l_i|l_0,\ldots,\hat{l}_i,\ldots,l_4)\Bigg)
\end{align}
The second part $\partial^{(2)}\circ\tau^3_{1,\varepsilon}(l^*_0,\ldots,l^*_4)$ is
\begin{align}\label{sec1q}
=-\frac{1}{3}\sum^4_{i=0}(-1)^i\Bigg(-&\frac{r_\varepsilon(l^*_i|l^*_0,\ldots,\hat{l}^*_i,\ldots,l^*_4)}{r(l_i|l_0,\ldots,\hat{l}_i,\ldots,l_4)}\wedge \frac{r_\varepsilon(l^*_i|l^*_0,\ldots,\hat{l}^*_i,\ldots,l^*_4)}{1-r(l_i|l_0,\ldots,\hat{l}_i,\ldots,l_4)}\notag\\
\wedge&\sum_{\substack{j=0 \\ j\neq i}}^4\left(\frac{\Delta(l^*_0,\ldots,\hat{l}^*_i,\ldots,\hat{l}^*_j,\ldots,l^*_4)_\varepsilon}{\Delta(l_0,\ldots,\hat{l}_i,\ldots,\hat{l}_j,\ldots,l_4)}\right)\Bigg)
\end{align}

then we calculate $\frac{b_\varepsilon}{a}$ and $\frac{b_\varepsilon}{1-a}$. i.e. all the values of the form  $\frac{r_\varepsilon(l^*_0|l^*_1,l^*_2,l^*_3,l^*_4)}{r(l_0|l_1,l_2,l_3,l_4)}$ and $\frac{r_\varepsilon(l^*_0|l^*_1,l^*_2,l^*_3,l^*_4)}{1-r(l_0|l_1,l_2,l_3,l_4)}$.
By using formula (\ref{reps1}) we have
\begin{align}
\frac{r_\varepsilon(l^*_0|l^*_1,l^*_2,l^*_3,l^*_4)}{r(l_0|l_1,l_2,l_3,l_4)}=\frac{(l^*_0l^*_1l^*_4)_\varepsilon}{(l_0l_1l_4)}+\frac{(l^*_0l^*_2l^*_3)_\varepsilon}{(l_0l_2l_3)}-\frac{(l^*_0l^*_2l^*_4)_\varepsilon}{(l_0l_2l_4)}-\frac{(l^*_0l^*_1l^*_3)_\varepsilon}{(l_0l_1l_3)}\notag
\end{align}
Similarly we can find this ratio for each value of $i=0,\ldots,4$. Now use formula (\ref{reps1}) as well as identities (\ref{2did}) and \eqref{newid}, we have
\[\frac{r_\varepsilon(l^*_0|l^*_1,l^*_2,l^*_3,l^*_4)}{1-r(l_0|l_1,l_2,l_3,l_4)}= \frac{(l^*_0l^*_2l^*_4)_\varepsilon
}{(l_0l_2l_4)}+\frac{(l^*_0l^*_1l^*_3)_\varepsilon}{(l_0l_1l_3)}-\frac{(l^*_0l^*_3l^*_4)_\varepsilon}{(l_0l_3l_4)}-\frac{(l^*_0l^*_1l^*_2)_\varepsilon}{(l_0l_1l_2)}\]
After calculating all these values. Expand the sums (\ref{first1q}) and (\ref{sec1q}) and put all values what we have calculated above. Let us talk about (\ref{first1q}). In this sum we have huge amount of terms, so we group them in a suitable way. First collect all the terms involving $\frac{(l^*_0l^*_1l^*_2)_\varepsilon}{(l_0l_1l_2)}\otimes \cdots$, we find that there are 6 different terms with coefficient -3 involving $\frac{(l^*_0l^*_1l^*_2)_\varepsilon}{(l_0l_1l_2)}\otimes \cdots$
\begin{align}
-3\frac{(l^*_0l^*_1l^*_2)_\varepsilon}{(l_0l_1l_2)}\otimes\Big(&(l_0l_1l_3)\wedge(l_1l_2l_3)+(l_0l_2l_4)\wedge(l_1l_2l_3)+(l_0l_1l_4)\wedge(l_0l_2l_4)\notag\\
-&(l_0l_1l_3)\wedge(l_0l_2l_3)-(l_0l_1l_4)\wedge(l_1l_2l_4)-(l_0l_2l_3)\wedge(l_1l_2l_3)\Big)\notag
\end{align}
There are exactly 10 possible terms of $\frac{(l^*_il^*_jl^*_k)_\varepsilon}{(l_il_jl_k)}$. Compute all of them individually. We will see that each will have the coefficient $-3$ that will be cancelled by $-\frac{1}{3}$ in (\ref{first1q}) and then combine 60 different terms with 6 in a group of same $\frac{(l^*_il^*_jl^*_k)_\varepsilon}{(l_il_jl_k)}$, write in the sum form then we will note that it will be the same as (\ref{first0q}). 

Computation for the second  part is relatively easy and direct. We need to put all values of the form $\frac{r_\varepsilon(l^*_0|l^*_1,l^*_2,l^*_3,l^*_4)}{r(l_0|l_1,l_2,l_3,l_4)}$ and $\frac{r_\varepsilon(l^*_0|l^*_1,l^*_2,l^*_3,l^*_4)}{1-r(l_0|l_1,l_2,l_3,l_4)}$ in (\ref{sec1q}), expand the sums, use $a\wedge a=0$ modulo 2 torsion. Here we will have simplified result which can be recombined in the sum notation which will be same as (\ref{sec0q}).
\end{proof}
\begin{thm}\label{claim4b}
The left square of the diagram (\ref{bicomp3b}), i.e. 
\begin{displaymath}
\xymatrix{
C_6(\varmathbb{A}^3_{F[\varepsilon]_2})\ar[d]^{\tau^3_{2,\varepsilon}}\ar[rr]^{d}&&C_5(\varmathbb{A}^3_{F[\varepsilon]_2})\ar[d]^{\tau^3_{1,\varepsilon}}\\
T\mathcal{B}_3(F)\ar[rr]^{\partial_\varepsilon\qquad\qquad}&&\left(T\mathcal{B}_2(F)\otimes F^\times\right)\oplus\left(F\otimes \mathcal{B}_2(F)\right)}
\end{displaymath}
is commutative i.e. $\tau^3_{2,\varepsilon}\circ \partial_\varepsilon=d\circ \tau^3_{1,\varepsilon}$
\end{thm}
\begin{proof}
The map $\tau^3_{2,\varepsilon}$ gives 720 terms and due to symmetry (cyclic and inverse) we find 120 different ones (up to inverse). We will use the same technique here which we have used in the proof of theorem \ref{claim4a}.
By definition, we have  
\[\tau^3_{2,\varepsilon}(l^*_0,\ldots,l^*_5)=\frac{2}{45}\text{Alt}_6\left\langle r_3(l_0,\ldots,l_5);r_{3,\varepsilon}(l^*_0,\ldots,l^*_5)\right]_3 \]
For convenience and similar to our previous conventions, we will abbreviate our notation by dropping $\Delta$ and commas.
\begin{align}\label{4b1}
&\partial_\varepsilon\circ\tau^3_{2\varepsilon}(l^*_0\ldots l^*_5)\notag\\
=&\frac{2}{45}\text{Alt}_6\left\lbrace\left\langle r_3(l_0\ldots l_5);r_{3,\varepsilon}(l^*_0\ldots l^*_5)\right]_2\otimes r_3(l_0\ldots l_5)+\frac{r_{3,\varepsilon}(l^*_0\ldots l^*_5)}{r_3(l_0\ldots l_5)}\otimes \left[r_3(l_0\ldots l_5)\right]_2\right\rbrace
\end{align}
We need to compute the value of $\frac{r_{3,\varepsilon}(l^*_0\ldots l^*_5)}{r_3(l_0\ldots l_5)}$ which is 
\[=\frac{(l^*_0l^*_1l^*_3)_\varepsilon}{(l_0l_1l_3)}+\frac{(l^*_1l^*_2l^*_4)_\varepsilon}{(l_1l_2l_4)}+\frac{(l^*_2l^*_0l^*_5)_\varepsilon}{(l_2l_0l_5)}-\frac{(l^*_0l^*_1l^*_4)_\varepsilon}{(l_0l_1l_4)}-\frac{(l^*_1l^*_2l^*_5)_\varepsilon}{(l_1l_2l_5)}-\frac{(l^*_2l^*_0l^*_3)_\varepsilon}{(l_2l_0l_3)}\]
(\ref{4b1}) can also be written as
\begin{align}
=&\frac{2}{45}\text{Alt}_6\Bigg\{\left\langle r_3(l_0\ldots l_5);r_{3,\varepsilon}(l^*_0\ldots l^*_5)\right]_2\otimes\frac{(l_0l_1l_3)(l_1l_2l_4)(l_2l_0l_5)}{(l_0l_1l_4)(l_1l_2l_5)(l_2l_0l_3)}\notag\\
+&\left(\frac{(l^*_0l^*_1l^*_3)_\varepsilon}{(l_0l_1l_3)}+\frac{(l^*_1l^*_2l^*_4)_\varepsilon}{(l_1l_2l_4)}+\frac{(l^*_2l^*_0l^*_5)_\varepsilon}{(l_2l_0l_5)}-\frac{(l^*_0l^*_1l^*_4)_\varepsilon}{(l_0l_1l_4)}-\frac{(l^*_1l^*_2l^*_5)_\varepsilon}{(l_1l_2l_5)}-\frac{(l^*_2l^*_0l^*_3)_\varepsilon}{(l_2l_0l_3)}\right)\otimes\left[r_3(l_0\ldots l_5)\right]_2\Bigg\}\notag
\end{align}
We will consider here only first part of the above relation.
\[\frac{2}{45}\text{Alt}_6\{\left\langle r_3(l_0\ldots l_5);r_{3,\varepsilon}(l^*_0\ldots l^*_5)\right]_2\otimes\frac{(l_0l_1l_3)(l_1l_2l_4)(l_2l_0l_5)}{(l_0l_1l_4)(l_1l_2l_5)(l_2l_0l_3)}\]
Further,
\begin{align}\label{4b2}
=&\text{Alt}_6\left\lbrace \left\langle r_3(l_0\ldots l_5);r_{3,\varepsilon}(l^*_0\ldots l^*_5)\right]_2\otimes (l_0l_1l_3)\right\rbrace\notag\\
+&\text{Alt}_6\left\lbrace \left\langle r_3(l_0\ldots l_5);r_{3,\varepsilon}(l^*_0\ldots l^*_5)\right]_2\otimes (l_1l_2l_4)\right\rbrace\notag\\
+&\text{Alt}_6\left\lbrace \left\langle r_3(l_0\ldots l_5);r_{3,\varepsilon}(l^*_0\ldots l^*_5)\right]_2\otimes (l_2l_0l_5)\right\rbrace\notag\\
-&\text{Alt}_6\left\lbrace \left\langle r_3(l_0\ldots l_5);r_{3,\varepsilon}(l^*_0\ldots l^*_5)\right]_2\otimes (l_0l_1l_4)\right\rbrace\notag\\
-&\text{Alt}_6\left\lbrace \left\langle r_3(l_0\ldots l_5);r_{3,\varepsilon}(l^*_0\ldots l^*_5)\right]_2\otimes (l_1l_2l_5)\right\rbrace\notag\\
-&\text{Alt}_6\left\lbrace \left\langle r_3(l_0\ldots l_5);r_{3,\varepsilon}(l^*_0\ldots l^*_5)\right]_2\otimes (l_2l_0l_3)\right\rbrace
\end{align}
We use the even cycle $(l_0l_1l_2)(l_3l_4l_5)$ ( or  $(l^*_0l^*_1l^*_2)(l^*_3l^*_4l^*_5)$) to obtain
\begin{align}
&\text{Alt}_6\left\lbrace \left\langle r_3(l_0l_1l_2l_3l_4l_5);r_{3,\varepsilon}(l^*_0l^*_1l^*_2l^*_3l^*_4l^*_5)\right]_2\otimes (l_0l_1l_3)\right\rbrace\notag\\
=&\text{Alt}_6\left\lbrace \left\langle r_3(l_1l_2l_0l_4l_5l_3);r_{3,\varepsilon}(l^*_1l^*_2l^*_0l^*_4l^*_5l^*_3)\right]_2\otimes (l_1l_2l_4)\right\rbrace\notag
\end{align}

We can also use here the symmetry 
\[\left\langle r_3(l_0l_1l_2l_3l_4l_5);r_{3,\varepsilon}(l^*_0l^*_1l^*_2l^*_3l^*_4l^*_5)\right]_2=\left\langle r_3(l_1l_2l_0l_4l_5l_3);r_{3,\varepsilon}(l^*_1l^*_2l^*_0l^*_4l^*_5l^*_3)\right]_2\]
since
\[r_{3,\varepsilon}(l^*_0l^*_1l^*_2l^*_3l^*_4l^*_5)=r_{3,\varepsilon}(l^*_1l^*_2l^*_0l^*_4l^*_5l^*_3)\quad\text{precisely both have the same factors}\]
and similar for the others as well so that (\ref{4b2}) can be written as
\begin{align}
=\frac{2}{15}\text{Alt}_6\Big\{ &\left\langle r_3(l_0l_1l_2l_3l_4l_5);r_{3,\varepsilon}(l^*_0l^*_1l^*_2l^*_3l^*_4l^*_5)\right]_2\otimes (l_0l_1l_3)\notag\\
-&\left\langle r_3(l_0l_1l_2l_3l_4l_5);r_{3,\varepsilon}(l^*_0l^*_1l^*_2l^*_3l^*_4l^*_5)\right]_2\otimes(l_0l_1l_4)\Big\}\notag
\end{align}
If we apply the odd permutation $(l_3l_4)$ (or $(l^*_3l^*_4)$), then we have
\[=\frac{2}{15}\cdot2\text{Alt}_6\left\lbrace \left\langle r_3(l_0l_1l_2l_3l_4l_5);r_{3,\varepsilon}(l^*_0l^*_1l^*_2l^*_3l^*_4l^*_5)\right]_2\otimes (l_0l_1l_3)\right\rbrace\]
Again apply an odd permutation $(l_0l_3)$ ( or $(l^*_0l^*_3)$)
\begin{align}
=\frac{2}{15}\text{Alt}_6\Big\{&\left\langle r_3(l_0l_1l_2l_3l_4l_5);r_{3,\varepsilon}(l^*_0l^*_1l^*_2l^*_3l^*_4l^*_5)\right]_2\otimes (l_0l_1l_3)\notag\\
-&\left\langle r_3(l_3l_1l_2l_0l_4l_5);r_{3,\varepsilon}(l^*_3l^*_1l^*_2l^*_0l^*_4l^*_5)\right]_2\otimes (l_3l_1l_0)\Big\}\notag
\end{align}
but up to 2-torsion, which we ignore here, we have $(l_0l_1l_3)=(l_3l_1l_0)$ and then the above will become
\begin{align}\label{4b21}
=\frac{2}{15}\text{Alt}_6\Big\{\Big(&\left\langle r_3(l_0l_1l_2l_3l_4l_5);r_{3,\varepsilon}(l^*_0l^*_1l^*_2l^*_3l^*_4l^*_5)\right]_2\notag\\
-&\left\langle r_3(l_3l_1l_2l_0l_4l_5);r_{3,\varepsilon}(l^*_3l^*_1l^*_2l^*_0l^*_4l^*_5)\right]_2\Big)\otimes (l_0l_1l_3)\Big\}
\end{align}
Recall from (ref of first required) that the triple-ratio \[r_3(l_0l_1l_2l_3l_4l_5)=\frac{(l_0l_1l_3)(l_1l_2l_4)(l_2l_0l_5)}{(l_0l_1l_4)(l_1l_2l_5)(l_2l_0l_3)}\] can be written as the ratio of two projected cross-ratios.

We will see here that $r_{3,\varepsilon}(l^*_0l^*_1l^*_2l^*_3l^*_4l^*_5)$ can also be converted into the ratio of two first order cross-ratios. 

Let $a$ and $b$ be two projected cross-ratios whose ratio is the triple-ratio \[r_3(l_0l_1l_2l_3l_4l_5)=\frac{(l_0l_1l_3)(l_1l_2l_4)(l_2l_0l_5)}{(l_0l_1l_4)(l_1l_2l_5)(l_2l_0l_3)}\]
then $r_{3,\varepsilon}(l^*_0l^*_1l^*_2l^*_3l^*_4l^*_5)$ will be written as $\left(\frac{a^*}{b^*}\right)_\varepsilon$. Since we can also write as
\[\textbf{r}_3(l^*_0l^*_1l^*_2l^*_3l^*_4l^*_5)=r_3(l_0l_1l_2l_3l_4l_5)+r_{3,\varepsilon}(l^*_0l^*_1l^*_2l^*_3l^*_4l^*_5)\varepsilon\]
or
\[\textbf{r}_3(l^*_0l^*_1l^*_2l^*_3l^*_4l^*_5)=r_3(l_0l_1l_2l_3l_4l_5)+\left(\textbf{r}_3(l^*_0l^*_1l^*_2l^*_3l^*_4l^*_5)\right)_\varepsilon\varepsilon\]
we get 
\[r_{3,\varepsilon}(l^*_0l^*_1l^*_2l^*_3l^*_4l^*_5)=\left(\frac{(l^*_0l^*_1l^*_3)(l^*_1l^*_2l^*_4)(l^*_2l^*_0l^*_5)}{(l^*_0l^*_1l^*_4)(l^*_1l^*_2l^*_5)(l^*_2l^*_0l^*_3)}\right)_\varepsilon\]
Now it is clear that $r_{3,\varepsilon}(l^*_0l^*_1l^*_2l^*_3l^*_4l^*_5)$ can also be written as the ratio or product of two projected cross-ratios. There are exactly three ways to write it (projected by $(l^*_0\text{ and }l^*_1)$, $(l^*_1\text{ and }l^*_2)$ and $(l^*_0\text{ and }l^*_2)$) but we will use here $l^*_1$ and $l^*_2$. The last expression can be written as
\[r_{3,\varepsilon}(l^*_0l^*_1l^*_2l^*_3l^*_4l^*_5)=\left(\frac{\textbf{r}(l^*_2|l^*_1l^*_0l^*_5l^*_3)}{\textbf{r}(l^*_1|l^*_0l^*_2l^*_3l^*_4)}\right)_\varepsilon\]
and (\ref{4b21}) can be written as
\begin{align}
=\frac{2}{15}\text{Alt}_6\Bigg\{&\left\langle \frac{r(l_2|l_1l_0l_5l_3)}{r(l_1|l_0l_2l_3l_4)};\left(\frac{\textbf{r}(l^*_2|l^*_1l^*_0l^*_5l^*_3)}{\textbf{r}(l^*_1|l^*_0l^*_2l^*_3l^*_4)}\right)_\varepsilon\right]_2\otimes (l_0l_1l_3)\notag\\
-&\left\langle\frac{r(l_2|l_1l_3l_5l_0)}{r(l_1|l_3l_2l_0l_4)};\left( \frac{\textbf{r}(l^*_2|l^*_1l^*_3l^*_5l^*_0)}{\textbf{r}(l^*_1|l^*_3l^*_2l^*_0l^*_4)}\right)_\varepsilon\right]_2\otimes (l_0l_1l_3)\Bigg\}\notag
\end{align}
Applying five-term relations in $T\mathcal{B}_2(F)$ which are analogous to the one in (\ref{t5term}).
\begin{align}\label{4b3}
=&\frac{2}{15}\text{Alt}_6\{\Bigg(\left\langle r(l_2|l_1l_0l_5l_3);r_\varepsilon(l^*_2|l^*_1l^*_0l^*_5l^*_3)\right]_2-\left\langle r(l_1|l_0l_2l_3l_4);r_\varepsilon(l^*_1|l^*_0l^*_2l^*_3l^*_4)\right]_2\notag\\
-&\left\langle \frac{r(l_2|l_1l_5l_3l_0)}{r(l_1|l_0l_3l_4l_2)};\left(\frac{\textbf{r}(l^*_2|l^*_1l^*_5l^*_3l^*_0)}{\textbf{r}(l^*_1|l^*_0l^*_3l^*_4l^*_2)}\right)_\varepsilon\right]_2\Bigg)\otimes (l_0l_1l_3)\}
\end{align}
For each individual determinant, e.g. $(l_0l_1l_3)$ will have three terms. First consider the third term of (\ref{4b3})
\begin{align}
&\frac{2}{15}\text{Alt}_6\left\lbrace \left\langle \frac{r(l_2|l_1l_5l_3l_0)}{r(l_1|l_0l_3l_4l_2)};\left(\frac{\textbf{r}(l^*_2|l^*_1l^*_5l^*_3l^*_0)}{\textbf{r}(l^*_1|l^*_0l^*_3l^*_4l^*_2)}\right)_\varepsilon\right]_2\otimes(l_0l_1l_3)\right\rbrace\notag\\
=&\frac{2}{15}\text{Alt}_6\left\lbrace \frac{1}{36}\text{Alt}_{(l_0l_1l_3)(l_2l_4l_5)}\left(\left\langle \frac{\textbf{r}(l_2|l_1l_5l_3l_0)}{r(l_1|l_0l_3l_4l_2)};\left(\frac{\textbf{r}(l^*_2|l^*_1l^*_5l^*_3l^*_0)}{\textbf{r}(l^*_1|l^*_0l^*_3l^*_4l^*_2)}\right)_\varepsilon\right]_2\otimes(l_0l_1l_3)\right)\right\rbrace\notag
\end{align}
We need a subgroup in $S_6$ which fixes $(l_0l_1l_3)$ as a determinant i.e. $(l_0l_1l_3)\sim(l_3l_1l_0)\sim(l_3l_0l_1)\cdots$

Here in this case $S_3$ permuting $\{l_0,l_1,l_3\}$ and another one permuting $\{l_2,l_4,l_5\}$ i.e. $S_3\times S_3$. Now consider
\begin{align}
&\text{Alt}_{(l_0l_1l_3)(l_2l_4l_5)}\left\lbrace \left\langle \frac{r(l_2|l_1l_5l_3l_0)}{r(l_1|l_0l_3l_4l_2)};\left(\frac{\textbf{r}(l^*_2|l^*_1l^*_5l^*_3l^*_0)}{\textbf{r}(l^*_1|l^*_0l^*_3l^*_4l^*_2)}\right)_\varepsilon\right]_2\otimes(l_0l_1l_3)\right\rbrace\notag\\
=&\text{Alt}_{(l_0l_1l_3)(l_2l_4l_5)}\left\lbrace \left\langle\frac{(l_2l_5l_3)(l_1l_0l_4)}{(l_2l_5l_0)(l_1l_3l_4)};\left(\frac{(l^*_2l^*_5l^*_3)(l^*_1l^*_0l^*_4)}{(l^*_2l^*_5l^*_0)(l^*_1l^*_3l^*_4)}\right)_\varepsilon\right]_2\otimes(l_0l_1l_3)\right\rbrace\notag
\end{align}
By using odd permutation $(l_2l_5)$ the above becomes zero.

then (\ref{4b3}) becomes
\begin{align}\label{4b4}
=&\frac{2}{15}\text{Alt}_6\left\lbrace \left(\left\langle r(l_2|l_1l_0l_5l_3);r_\varepsilon(l^*_2|l^*_1l^*_0l^*_5l^*_3)\right]_2-\left\langle r(l_1|l_0l_2l_3l_4);r_\varepsilon(l^*_1|l^*_0l^*_2l^*_3l^*_4)\right]_2\right)\otimes (l_0l_1l_3)\right\rbrace
\end{align}
Consider the first term now,
\[\frac{2}{15}\text{Alt}_6\left\lbrace \left\langle r(l_2|l_1l_0l_5l_3);r_\varepsilon(l^*_2|l^*_1l^*_0l^*_5l^*_3)\right]_2\otimes (l_0l_1l_3)\right\rbrace\]
\[=\frac{2}{15}\text{Alt}_6\left\lbrace\frac{1}{36}\text{Alt}_{(l_0l_1l_3)(l_2l_4l_5)}\left\lbrace \left\langle r(l_2|l_1l_0l_5l_3);r_\varepsilon(l^*_2|l^*_1l^*_0l^*_5l^*_3)\right]_2\otimes (l_0l_1l_3)\right\rbrace\right\rbrace\]
The permutation $(l_0l_2l_3)$ does not have any role because the ratio is projected by 2. So, it will be reduced to $S_3$.
\[=\frac{2}{15}\text{Alt}_6\left\lbrace\frac{1}{6}\text{Alt}_{(l_2l_4l_5)}\left\lbrace \left\langle r(l_2|l_1l_0l_5l_3);r_\varepsilon(l^*_2|l^*_1l^*_0l^*_5l^*_3)\right]_2\otimes (l_0l_1l_3)\right\rbrace\right\rbrace\]
Write all possible  inner alternation, then
\begin{align}
=\frac{1}{45}\text{Alt}_6\Bigg\{\Bigg(&\left\langle r(l_4|l_1l_0l_2l_3);r_\varepsilon(l^*_4|l^*_1l^*_0l^*_2l^*_3)\right]_2-\left\langle r(l_2|l_1l_0l_4l_3);r_\varepsilon(l^*_2|l^*_1l^*_0l^*_4l^*_3)\right]_2\notag\\ 
+&\left\langle r(l_5|l_1l_0l_4l_3);r_\varepsilon(l^*_5|l^*_1l^*_0l^*_4l^*_3)\right]_2-\left\langle r(l_4|l_1l_0l_5l_3);r_\varepsilon(l^*_4|l^*_1l^*_0l^*_5l^*_3)\right]_2\notag\\
+&\left\langle r(l_2|l_1l_0l_5l_3);r_\varepsilon(l^*_2|l^*_1l^*_0l^*_5l^*_3)\right]_2-\left\langle r(l_5|l_1l_0l_2l_3);r_\varepsilon(l^*_5|l^*_1l^*_0l^*_2l^*_3)\right]_2\Bigg)\otimes (l_0l_1l_3)\Bigg\}\notag
\end{align}
Now we can use projected five-term relation for $T\mathcal{B}_2(F)$ here, 
\begin{align}\label{4b5}
=\frac{1}{45}\text{Alt}_6\Bigg\{\Bigg(&\left\langle r(l_0|l_1l_2l_3l_4);r_\varepsilon(l^*_0|l^*_1l^*_2l^*_3l^*_4)\right]_2-\left\langle r(l_1|l_0l_2l_3l_4);r_\varepsilon(l^*_1|l^*_0l^*_2l^*_3l^*_4)\right]_2\notag\\
-&\left\langle r(l_3|l_0l_1l_2l_4);r_\varepsilon(l^*_3|l^*_0l^*_1l^*_2l^*_4)\right]_2+\left\langle r(l_0|l_1l_4l_3l_5);r_\varepsilon(l^*_0|l^*_1l^*_4l^*_3l^*_5)\right]_2\notag\\
-&\left\langle r(l_1|l_0l_4l_3l_5);r_\varepsilon(l^*_1|l^*_0l^*_4l^*_3l^*_5)\right]_2-\left\langle r(l_3|l_0l_1l_4l_5);r_\varepsilon(l^*_3|l^*_0l^*_1l^*_4l^*_5)\right]_2\notag\\
+&\left\langle r(l_0|l_1l_5l_3l_2);r_\varepsilon(l^*_0|l^*_1l^*_5l^*_3l^*_2)\right]_2-\left\langle r(l_1|l_0l_5l_3l_2);r_\varepsilon(l^*_1|l^*_0l^*_5l^*_3l^*_2)\right]_2\notag\\
-&\left\langle r(l_3|l_0l_1l_5l_2);r_\varepsilon(l^*_3|l^*_0l^*_1l^*_5l^*_2)\right]_2\Bigg)\otimes (l_0l_1l_3)\Bigg\}\notag\\
&\text{Use the cycle $(l_0l_1l_3)(l_2l_4l_5)$ then we get}\notag\\
=\frac{1}{45}\cdot9\text{Alt}_6&\left\lbrace\left\langle r(l_0|l_1l_2l_3l_4);r_\varepsilon(l^*_0|l^*_1l^*_2l^*_3l^*_4)\right]_2\otimes (l_0l_1l_3)\right\rbrace
\end{align}
The second term of (\ref{4b4}) can be written as 
\[\frac{1}{45}\cdot-6\text{Alt}_6\left\lbrace\left\langle r(l_1|l_0l_2l_3l_4);r_\varepsilon(l^*_1|l^*_0l^*_2l^*_3l^*_4)\right]_2\otimes (l_0l_1l_3)\right\rbrace\]
(\ref{4b5}) can be combined with the above so we get
\begin{align}
=\frac{1}{45}\text{Alt}_6\left\lbrace\left(9\left\langle r(l_0|l_1l_2l_3l_4);r_\varepsilon(l^*_0|l^*_1l^*_2l^*_3l^*_4)\right]_2-6\left\langle r(l_1|l_0l_2l_3l_4);r_\varepsilon(l^*_1|l^*_0l^*_2l^*_3l^*_4)\right]_2\right)\otimes (l_0l_1l_3)\right\rbrace
\end{align}
Use the permutation $(l_0l_1l_3)(l_2l_4l_5)$ to get
\[=\frac{1}{3}\text{Alt}_6\left\lbrace\left\langle r(l_0|l_1l_2l_3l_4);r_\varepsilon(l^*_0|l^*_1l^*_2l^*_3l^*_4)\right]_2\otimes (l_0l_1l_3)\right\rbrace\]
Since $\mathcal{B}_2(F)$ satisfies five-term relation then we can write the following.
\begin{align}\label{4b6}
=\frac{1}{3}\text{Alt}_6\left\lbrace \left\langle r(l_0|l_1l_2l_3l_4);r_\varepsilon(l^*_0|l^*_1l^*_2l^*_3l^*_4)\right]_2\otimes (l_0l_1l_3)+\frac{(l^*_0l^*_1l^*_3)_\varepsilon}{(l_0l_1l_3)}\otimes\left[r(l_0|l_1l_2l_3l_4)\right]_2\right\rbrace
\end{align}

Now go to the other side. Map $\tau^3_{1,\varepsilon}$ can also be written in the alternation sum form
\begin{align}
\tau^3_{1,\varepsilon}(l^*_0\ldots l^*_4)=\frac{1}{3}\text{Alt}\{&\left\langle r(l_0|l_1l_2l_3l_4);r_\varepsilon(l^*_0|l^*_1l^*_2l^*_3l^*_4)\right]_2\otimes(l_0l_1l_2)\notag\\
+&\frac{(l^*_0l^*_1l^*_2)_\varepsilon}{(l_0l_1l_2)}\otimes [r(l_0|l_1l_2l_3l_4)]_2\}\notag
\end{align}
Compute $\tau^3_{1,\varepsilon}\circ d(l^*_0\ldots l^*_5)$ and apply cycle $(l_0l_1l_2l_3l_4l_5)$ for $d$ and then expand Alt$_5$ from the definition of $\tau^3_{1,\varepsilon}$ so we get
\begin{align}
\tau^3_{1\varepsilon}\circ d(l^*_0\ldots l^*_5)=\frac{1}{3}\text{Alt}_6\{&\left\langle r(l_0|l_1l_2l_3l_4);r_\varepsilon(l^*_0|l^*_1l^*_2l^*_3l^*_4)\right]_2\otimes(l_0l_1l_2)\notag\\
+&\frac{(l^*_0l^*_1l^*_2)_\varepsilon}{(l_0l_1l_2)}\otimes [r(l_0|l_1l_2l_3l_4)]_2\}\notag
\end{align}
Use the odd permutation $(l_2l_3)$, then
\begin{align}
=-\frac{1}{3}\text{Alt}_6\{&\left\langle r(l_0|l_1l_3l_2l_4);r_\varepsilon(l^*_0|l^*_1l^*_3l^*_2l^*_4)\right]_2\otimes(l_0l_1l_3)\notag\\
+&\frac{(l^*_0l^*_1l^*_3)_\varepsilon}{(l_0l_1l_3)}\otimes [r(l_0|l_1l_3l_2l_4)]_2\}\notag
\end{align}
Finally use two-term relation in $T\mathcal{B}_2(F)$ and $\mathcal{B}_2(F)$ to get the correct sign. The final answer will be the same as (\ref{4b6})
\end{proof}

If we combine Theorem \ref{claim4a} and \ref{claim4b}, then we see that the diagram \eqref{bicomp3b} is commutative and have maps of morphisms between the Grassmannian complex and the tangential complex for weight 3. Here we have some results 
\begin{prop}\label{zero3}
The map $C_5(\varmathbb{A}^4_{F[\varepsilon]_2})\xrightarrow{d'}C_4(\varmathbb{A}^3_{F[\varepsilon]_2})\xrightarrow{\tau^3_{0,\varepsilon}}\left(F\otimes \bigwedge^2 F^\times\right)\oplus \left(\bigwedge^3F\right)$ is zero.
\end{prop}
\begin{proof}
The proof of this lemma is direct by calculation. Let $(l^*_0,\ldots,l^*_4) \in C_5(\varmathbb{A}^4_{F[\varepsilon]_2})$
Where 
\[l^*_i=\left(\begin{array}{c} a+a_\varepsilon\varepsilon\\b+b_\varepsilon\varepsilon\\c+c_\varepsilon\varepsilon\\d+d_\varepsilon\varepsilon \end{array}                                                                                                                                          \right)=\left(\begin{array}{c}a\\b\\c\\d\end{array}\right)+\left(\begin{array}{c}a_\varepsilon\\b_\varepsilon\\c_\varepsilon\\d_\varepsilon\varepsilon \end{array}\right)=l_i+l_{i\varepsilon}\varepsilon\]
Let $\omega$ be the volume formed in four-dimensional vector space, and $\Delta(l_i,\cdot,\cdot,\cdot)$ be the volume form in $V_4/\langle l_i \rangle$.
\begin{align}
&\tau^3_{0,\varepsilon}\circ d'(l^*_0,\ldots,l^*_4)\notag\\
=&\tau^3_{0,\varepsilon}\Bigg(\sum^4_{i=0}(-1)^i(l^*_i|l^*_0,\ldots,\hat{l}^*_i,\ldots,l^*_4)\Bigg)\notag\\
&\text{Consider the first coordinate of the map first}\notag\\
=&\widetilde{\text{Alt}}_{(01234)}\Bigg(\sum_{i=0}^3(-1)^{i}\Big(\frac{\Delta(l^*_0,\ldots,\hat{l}^*_i,\ldots,l^*_3,l^*_4)_\varepsilon}{\Delta(l_0,\ldots,\hat{l}_i,\ldots,l_3,l_4)}\otimes\frac{\Delta(l_0,\ldots,\hat{l}_{i+1},\ldots,l_3,l_4)}{\Delta(l_0,\ldots,\hat{l}_{i+2},\ldots,l_3,l_4)}\notag\\
&\quad\quad\quad\wedge\frac{\Delta(l_0,\ldots,\hat{l}_{i+3},\ldots,l_3,l_4)}{\Delta(l_0,\ldots,\hat{l}_{i+2},\ldots,l_3,l_4)}\Big)\quad i\mod 4\Bigg)
\end{align}
First, we expand inner sum which gives us 12 different terms after simplification. When we apply alternation sum then we get 60 terms and there is direct cancellation which leads to zero. Now consider the second coordinate , which gives us
\[\widetilde{\text{Alt}}_{(01234)}\left(\sum^3_{i=0}(-1)^i\bigwedge^3_{\substack{j=0\\ j\neq i}}\frac{\Delta(l^*_0,\ldots,\hat{l}^*_j,\ldots,l^*_3,l^*_4)_\varepsilon}{\Delta(l_0,\ldots,\hat{l}_j,\ldots,l_3,l_4)}\right)\]
Again if we expand inner sum first, then we get only four different terms but after the application of alternation we get zero.
\end{proof}

As an analogy of Proposition \ref{zero3} in higher weight, we present the following result
\begin{prop}\label{zeron}
The map $C_{n+2}(\varmathbb{A}^{n+1}_{F[\varepsilon]_2})\xrightarrow{d'}C_{n+1}(\varmathbb{A}^{n}_{F[\varepsilon]_2})\xrightarrow{\tau^n_{0,\varepsilon}}\left(F\otimes \bigwedge^{n-1} F^\times\right)\oplus\left(\bigwedge^nF\right)$ is zero, where
\begin{align}
\tau^n_{0,\varepsilon}(l^*_0,\ldots,l^*_{n})\notag\\
=\sum_{i=0}^{n}(-1)^{i}\Bigg(&\frac{\Delta(l^*_0,\ldots,\hat{l}^*_i,\ldots,l^*_{n})_\varepsilon}{\Delta(l_0,\ldots,\hat{l}_i,\ldots,l_{n})}\otimes\frac{\Delta(l_0,\ldots,\hat{l}_{i+1},\ldots,l_{n})}{\Delta(l_0,\ldots,\hat{l}_{i+2},\ldots,l_{n})}\notag\\
&\wedge\cdots\wedge\frac{\Delta(l_0,\ldots,\hat{l}_{i+(n-1)},\ldots,l_{n})}{\Delta(l_0,\ldots,\hat{l}_{i+n},\ldots,l_{n})}\Bigg)+\Bigg(
\bigwedge_{\substack{j=0\\j\neq i}}^{n}\frac{\Delta(l^*_0,\ldots,\hat{l}^*_j,\ldots,l^*_{n})_\varepsilon}{\Delta(l_0,\ldots,\hat{l}_j,\ldots,l_{n})}\Bigg),\notag\\
&\qquad\qquad\qquad\qquad\qquad\qquad\qquad\qquad\qquad\quad i\mod (n+1)\notag
\end{align}
\end{prop}
\begin{proof}
Let $(l^*_0,\ldots,l^*_{n+1}) \in C_{n+2}(\varmathbb{A}^{n+1}_{F[\varepsilon]_2})$. We have
\[\tau^n_{0,\varepsilon}\circ d'(l^*_0,\ldots,l^*_{n+1})=\tau^n_{0,\varepsilon}\left(\sum^{n}_{i=0}(-1)^i(l^*_i|l^*_0,\ldots,\hat{l}^*_i,\ldots,l^*_{n+1})\right)\]
Now use definition of alternation to represent this sum then we have
\begin{align}\label{eq_zeron_1}
&\tau^n_{0,\varepsilon}\circ d'(l^*_0,\ldots,l^*_{n+1})\notag\\
=&\widetilde{\text{Alt}}_{(0\cdots n+1)}\Bigg\{\sum_{i=0}^{n}(-1)^{i}\Bigg(\Big(\frac{\Delta(l^*_0,\ldots,\hat{l}^*_i,\ldots,l^*_{n},l^*_{n+1})_\varepsilon}{\Delta(l_0,\ldots,\hat{l}_i,\ldots,l_{n},l_{n+1})}\otimes\frac{\Delta(l_0,\ldots,\hat{l}_{i+1},\ldots,l_{n},l_{n+1})}{\Delta(l_0,\ldots,\hat{l}_{i+2},\ldots,l_{n},l_{n+1})}\notag\\
&\wedge\cdots\wedge\frac{\Delta(l_0,\ldots,\hat{l}_{i+(n-1)},\ldots,l_{n},l_{n+1})}{\Delta(l_0,\ldots,\hat{l}_{i+n},\ldots,l_{n},l_{n+1})}\Big)+\Big(
\bigwedge_{\substack{j=0\\j\neq i}}^{n}\frac{\Delta(l^*_0,\ldots,\hat{l}^*_j,\ldots,l^*_{n},l^*_{n+1})_\varepsilon}{\Delta(l_0,\ldots,\hat{l}_j,\ldots,l_{n},l_{n+1})}\Big)\Bigg),\notag\\
&\qquad\qquad\qquad\qquad\qquad\qquad\qquad\qquad\qquad\qquad\qquad\quad i\mod n+1\Bigg\}
\end{align}
First expand the inner sum on first term that gives $n+1$ number of terms. Expand again by using the properties of wedge that gives $n(n+1)$ terms. Apply alternation sum on that gives us $n(n+1)(n+2)$ terms, so there are $n+2$ sets each consisting $n(n+1)$ terms and each term in $n(n+1)$ term  has $n+1$ sets of $n$ terms and good thing is that they cancelled set by set to give zero.

Now expand the inner sum in the second term of (\ref{eq_zeron_1}) that gives $n+1$ terms and then apply alternation sum which gives $n+2$ sets of $n+1$ terms, we find cancellation in the expansion of sum accordingly which gives zero as well.
\end{proof}

\section*{Acknowledgements:}
This article consists on a chapter of author's doctoral thesis written under the supervision of Dr. Herbert Gangl at University of Durham. Author would also like to thanks to Prof. Spencer Bloch for his comments and valuable suggestions.


\begin{thebibliography}{100}

\bibitem{BandE1} Bloch, S. and Esnault, H., \emph{The additive dilogarithm}, Kazuya Kato's fiftieth birthday, Doc. Math. (2003), Extra Vol. 131-155.


\bibitem{Cath1} Cathelineau, J-L., \emph{Infinitesimal Polylogarithms, multiplicative Presentations of K\"{a}hler Differentials and Goncharov complexes}, talk at the workshop on polylogarthms, Essen, May 1-4(1997).

\bibitem{Cath2} Cathelineau, J-L., \emph{Remarques sur les Diff\'{e}rentielles des Polylogarithmes Uniformes}, Ann. Inst. Fourier, Grenoble {\bf46}, (1996)1327-1347.

\bibitem{Cath3} Cathelineau, J-L., \emph{The tangent complex to the Bloch-Suslin complex}, Bull. Soc. Math. France {\bf135} (2007) 565-597

\bibitem{Cath4} Cathelineau, J-L., \emph{Projective Configurations, Homology of Orthogonal Groups, and Milnor K-theory}, Duke Mathematical Journal, {\bf2}(121), 2004

\bibitem{PandG}Elbaz-Vincent, Ph., and Gangl, H, \emph{On Poly(ana)logs I}, Compositio Mathematica, \textbf{130}, 161-210 (2002).


\bibitem{Gonc} Goncharov, A. B., \emph{Geometry of Configurations, Polylogarithms and Motivic Cohomology}, Adv. Math., {\bf114}(1995) 197-318.

\bibitem{Gonc1} Goncharov, A. B., \emph{Polylogarithms and Motivic Galois Groups}, Proceedings of the Seattle conf. on motives, Seattle July 1991, AMS Proceedings of Symposia in Pure Mathematics 2, {\bf55}(1994) 43-96.

\bibitem{Gonc2} Goncharov, A. B., \emph{Explicit construction of characteristic classes}, Advances in Soviet Mathematics, I. M. Gelfand Seminar 1, {\bf16}(1993) 169-210

\bibitem{Gonc3} Goncharov, A. B., \emph{Euclidean Scissors congruence groups and mixed Tate motives over dual numbers}, Mathematical Research Letters {\bf11} (2004) 771-784.

\bibitem{Gonc4} Goncharov, A, B., \emph{Deninger's conjecture on L-functions of elliptic curves at $s=3$}, J. Math. Sci. {\bf81} (1996), N3, 2631-2656, alg-geom/9512016. MR 1420221 (98c:19002)

\bibitem{GoncZ1} Goncharov, A. B. and Zhao, J., \emph{Grassmannian Trilogarithms}, Compositio Mathematica, \textbf{127}, 83-108, (2001)

\bibitem{KoSe} Ko\v{s}ir, T., Sethuraman, B.A., \emph{Determinantal varieties over truncated polynomial rings}, Journal of Pure and Applied Algebra {\bf195}(2005), 75-95 

\bibitem{LeNi} Leutbecher, A., Niklasch, G., \emph{On cliques of exceptional units and Lenstra's construction of Euclidean fields}, Lecture Notes in Mathematics, Vol. 1380/1989 (1989) 150-178, DOI 10.1007/BFboo86551.





\bibitem{SiDs} Siddiqui, R., \emph{Configuration complexes and a variant of Cathelineau's complex in weight 3}, arXiv:1205.3864 [math.NT], (2012)

\bibitem{Sieg} Siegel, C.L., \emph{Approximation algebraischer Zahlen},
Mathem. Ze/tschr. {\bf10} (1921), 173-213.



\bibitem{UnverS} \"{U}nver, S., \emph{On the additive dilogarithm}, Algebra Number Theory 3(1), 1-34(2009)

\bibitem{UnverS1} \"{U}nver, S., \emph{Additive polylogarithms and their functional equations}, Math. Ann. DOI 10.1007/s00208-010-0493-7, (2010)


\bibitem{Zhao1} Zhao, J., \emph{Motivic Complexes of Weight Three and Pairs of Simplices in Projective 3-Space}, Advances in Mathematics, \textbf{161}, 141-208 (2001)

\end{thebibliography}
\end{document}